\documentclass{amsproc}
\usepackage{amsfonts}
\usepackage{xcolor}

\theoremstyle{plain}

\newtheorem{corollary}{Corollary}

\newtheorem{definition}{Definition}
\newtheorem{example}{Example}

\newtheorem{lemma}{Lemma}
\newtheorem{notation}{Notation}

\newtheorem{proposition}{Proposition}
\newtheorem{remark}{Remark}

\newtheorem{theorem}{Theorem}
\numberwithin{equation}{section}

\newcommand{\comment}[1]{{}}
\DeclareMathOperator{\essinf}{ess\  inf}
\DeclareMathOperator{\esssup}{ess\ sup}

\begin{document}
\title[Assouad dimensions of random measures]{Assouad-like dimensions of
random Moran measures}
\author{Kathryn E. Hare}
\address{Dept. of Pure Mathematics, University of Waterloo, Waterloo, Ont.,
Canada, N2L 3G1}
\email{kehare@uwaterloo.ca}
\author{Franklin Mendivil}
\address{Department of Mathematics and Statistics, Acadia University,
Wolfville, N.S. Canada, B4P 2R6}
\email{franklin.mendivil@acadiau.ca}
\thanks{The research of K. Hare is partially supported by NSERC 2016:03719.
The research of F. Mendivil is partially supported by NSERC\ 2019:05237.}
\subjclass[2010]{Primary 28A80; Secondary 28C15, 60G57}
\keywords{ random Moran set, random measure, Assouad dimension,
quasi-Assouad dimension}
\thanks{This paper is in final form and no version of it will be submitted
for publication elsewhere.}

\begin{abstract}
In this paper, we determine the almost sure values of the $\Phi $-dimensions
of random measures supported on random Moran sets that satisfy a uniform
separation condition. The $\Phi $-dimensions are intermediate Assouad-like
dimensions, the (quasi-)Assouad dimensions and $\theta $-Assouad spectrum
being special cases. Their values depend on the size of $\Phi $, with one
size coinciding with the Assouad dimension and the other coinciding with the
quasi-Assouad dimension. We give many applications, including to
equicontractive self-similar measures and $1$-variable random Moran measures
such as Cantor-like measures with probabilities that are uniformly
distributed. We can also deduce the $\Phi $-dimensions of the underlying
random sets.
\end{abstract}

\maketitle

\section{Introduction}

Many notions of dimension have been developed to quantify the size of sets
and measures, in part, to better understand the geometry of the set and the
nature of the measure. Perhaps the most well known are Hausdorff and box
dimensions, global notions of size. More recently, there has been much
interest in understanding the local complexity of sets and measures and for
this various dimensions have been introduced including the (upper and lower)
Assouad dimensions, which quantify the `thickest' and `thinnest' parts of
sets and measures (see \cite{A2, FH, KLV, L1}), the less extreme
quasi-Assouad dimensions (\cite{CDW, HHT, HT, LX}), the $\theta $-Assouad
spectrum (\cite{FY}), and the most general, intermediate $\Phi $-dimensions (%
\cite{FY, GHM, HH}). The $\Phi $-dimensions range between the box and
Assouad dimensions. They are localized, like the Assouad dimensions, but
vary in the depth of the scales considered, thus they provide very refined
information. In this paper, we continue the study of the upper and lower $%
\Phi $-dimensions of measures, with a focus on random measures supported on
random Moran sets.

There is a long history of the study of dimensional properties of random
sets, with early important early papers including \cite{Fa, Gr}, for
example. More recently, Assouad-like dimensions of random sets have been
investigated in papers such as \cite{FMT, FT, GHM2, Tr, Tr2}.

By a random Moran measure, we will mean a probability measure supported on a
random Moran set in $\mathbb{R}^{D}$ that can be thought of as arising from
a random homogenous model of uncountably many equicontractive iterated
function systems satisfying a suitable uniform strong separation condition.
The number of children at each level, the ratios of child to parent
diameters and the probability weights assigned to each child that specify
the measure, are all to be iid random variables. We refer the reader to
Section \ref{setup} for the technical details.

The values of the $\Phi $-dimensions of these random measures depend on how
the dimension function $\Phi $ compares with the function $\Psi (x)=\log
|\log x|/|\log x|$ near $0$, as was also seen to be the case in \cite{GHM2}
where the $\Phi $-dimensions of random rearrangements of Cantor-like sets
were considered. Under mild technical assumptions on the probability and
ratio distributions, when $\Phi (x)\gg \Psi (x)$\footnote{%
We will write $f\gg g$ if there is a function $A$ and $\delta >0$ such that $%
f(x)\geq A(x)g(x)$ for all $0<x<\delta $ and $A(x)\rightarrow \infty $ as $%
x\rightarrow 0^{+}$.} (such as for the quasi-Assouad dimensions), then
almost surely the value of all the upper $\Phi $-dimensions of the random
measure is $\mathbb{E}(\log m)/\mathbb{E}(\log r)$ where $m$ is the minimum
probability and $r$ is the child/parent diameter ratio. For the lower $\Phi $%
-dimension we simply replace the minimum probability $m$ by the maximum $M$.
When $\Phi (x)\ll \Psi (x)$ (such as for the Assouad dimensions) and the
essential infimum of either the ratios or the probabilities is bounded away
from $0,$ then the upper (and lower) $\Phi $-dimensions again coincide,
being the almost sure extreme behaviour of $\log m/\log r$ (resp., $\log
M/\log r$). These results are formally stated and proven in Sections \ref%
{largesection} and \ref{smallsection}.

One special case to which our theorems apply is when the set is
deterministic and the probabilities are chosen uniformly distributed. For
example, take the classical middle-third Cantor set. If the two
probabilities, $p,1-p,$ are equal, all the dimensions (of either the set or
the measure) equal $\log 2/\log 3$. In contrast, if the probabilities are
chosen with $p$ uniformly distributed over $(0,1),$ then for $\Phi \gg \Psi
, $ the upper $\Phi $-dimension of the random measure is almost surely $%
(1+\log 2)/\log 3,$ while the lower dimension is $(1-\log 2)/\log 3$ (which,
perhaps surprisingly, do not average to $\log 2/\log 3$). For $\Phi \ll \Psi 
$ the upper and lower $\Phi $-dimensions are almost surely $\infty $ and $0$
respectively. More complicated formulas hold if the Cantor set has $T$
children, with $T>2,$ and the probabilities are chosen uniformly over the
simplex $\{(p_{i})_{i=1}^{T}:\sum_{i=1}^{T}p_{i}=1,p_{i}\geq 0\}$.

Another special case to which our theorems apply is a random 1-variable
(homogeneous) model with finitely many equicontractive iterated function
systems, satisfying the strong separation condition. In \cite{Tr} it was
shown that the upper quasi-Assouad dimensions for these random sets coincide
almost surely with their Hausdorff dimensions. If we take the measure with
uniform probabilities, the $\Phi $-dimensions of the measure equal those of
the random set, thus the Hausdorff dimension coincides almost surely with
the $\Phi $-dimensions for all $\Phi \gg \Psi $ (including both the upper
and lower quasi-Assouad dimensions). A different formula applies for $\Phi
\ll \Psi ,$ (including the Assouad dimension).

These examples are discussed in Section \ref{applications}, along with
others. Section \ref{prelim} contains the definitions of the $\Phi $%
-dimensions, as well as their basic properties, and details the random setup.

\section{Preliminaries\label{prelim}}

\subsection{Dimensions of sets and measures}

Given a bounded metric space $X,$ we denote the open ball centred at $x\in X$
and radius $R$ by $B(x,R)$. By a measure, we will always mean a Borel
probability measure on the metric space $X$.

\begin{definition}
By a \textbf{dimension function}, we mean a map $\Phi :(0,1)\rightarrow 
\mathbb{R}^{+}$ such that $x^{1+\Phi (x)}$ decreases to $0$ as $x$ decreases
to $0$.
\end{definition}

Interesting examples include the constant functions $\Phi (x)=\delta \geq 0$%
, $\Phi (x)=1/|\log x|$ and $\Phi (x)=\log \left\vert \log x\right\vert
/\left\vert \log x\right\vert .$

\begin{definition}
Let $\Phi $ be a dimension function and let $\mu $ be a measure on $X$. The 
\textbf{upper and lower }$\Phi $\textbf{-dimensions} of $\mu $ are given by%
\begin{equation*}
\overline{\dim }_{\Phi }\mu =\inf \left\{ 
\begin{array}{c}
d:(\exists C_{1},C_{2}>0)(\forall 0<r<R^{1+\Phi (R)}\leq R\leq C_{1}) \\ 
\frac{\mu (B(x,R))}{\mu (B(x,r))}\leq C_{2}\left( \frac{R}{r}\right) ^{d}%
\text{ }\forall x\in \text{supp}\mu%
\end{array}%
\right\}
\end{equation*}%
and 
\begin{equation*}
\underline{\dim }_{\Phi }\mu =\sup \left\{ 
\begin{array}{c}
d:(\exists C_{1},C_{2}>0)(\forall 0<r<R^{1+\Phi (R)}\leq R\leq C_{1}) \\ 
\frac{\mu (B(x,R))}{\mu (B(x,r))}\geq C_{2}\left( \frac{R}{r}\right) ^{d}%
\text{ }\forall x\in \text{supp}\mu%
\end{array}%
\right\}
\end{equation*}
\end{definition}

\begin{remark}
(i) The \textbf{upper and lower Assouad dimensions} of $\mu $ (also known as
the upper and lower regularity dimensions) studied by K\"{a}enm\"{a}ki et al
in \cite{KL, KLV} and Fraser and Howroyd in \cite{FH}, and denoted $\dim
_{A}\mu $ and $\dim _{L}\mu $ respectively, are the upper and lower $\Phi $%
-dimensions with $\Phi $ the constant function $0$. It is well known that a
measure $\mu $ is doubling if and only if $\dim _{A}\mu $ $<\infty $ and
uniformly perfect if and only if $\dim _{L}\mu >0$.

(ii) If we let $\Phi _{\theta }=1/\theta -1,$ then $\overline{\dim }_{\Phi
_{\theta }}\mu $ and \underline{$\dim $}$_{\Phi _{\theta }}\mu $ are
(basically) the \textbf{upper} and \textbf{lower $\theta $-Assouad spectrum}
introduced in \cite{FY}. The \textbf{upper} and \textbf{lower quasi-Assouad
dimensions} of $\mu $ developed in \cite{HHT, HT} are given by 
\begin{equation*}
\dim _{qA}\mu =\lim_{\theta \rightarrow 1}\overline{\dim }_{\Phi _{\theta
}}\mu \text{, }\dim _{qL}\mu =\lim_{\theta \rightarrow 1}\underline{\dim }%
_{\Phi _{\theta }}\mu .
\end{equation*}

As noted in \cite{HH}, there are always dimension functions which give rise
to the quasi-Assouad dimensions, but these need to be tailored to the
particular measure.

(iii) The \textbf{upper Minkowski dimension}, $\overline{\dim }_{M}\mu ,$
and the \textbf{Frostman dimension}, $\dim _{F}\mu ,$ coincide with the
upper and lower $\Phi $-dimensions respectively for $\Phi \rightarrow \infty 
$; see \cite{FK} and Proposition \ref{basic}.
\end{remark}

The upper and lower $\Phi $-dimensions of a measure were introduced in \cite%
{HH} to provide more refined information about the local behaviour of a
measure than that given by the upper and lower Assouad dimensions. They were
motivated, in part, by related definitions for dimensions of sets. To recall
these, we use the notation $N_{r}(Y)$ to mean the least number of balls of
radius $r$ that cover $Y\subseteq X$.

\begin{definition}
The \textbf{upper }and \textbf{lower }$\Phi $\textbf{-dimensions\ }of $%
E\subseteq X$ are given by 
\begin{equation*}
\overline{\dim }_{\Phi }E=\inf \left\{ 
\begin{array}{c}
\alpha :(\exists C_{1},C_{2}>0)(\forall 0<r\leq R^{1+\Phi (R)}\leq R<C_{1})%
\text{ } \\ 
N_{r}(B(z,R)\bigcap E)\leq C_{2}\left( \frac{R}{r}\right) ^{\alpha }\text{ }%
\forall z\in E%
\end{array}%
\right\}
\end{equation*}%
and 
\begin{equation*}
\underline{\dim }_{\Phi }E=\sup \left\{ 
\begin{array}{c}
\alpha :(\exists C_{1},C_{2}>0)(\forall 0<r\leq R^{1+\Phi (R)}\leq R<C_{1})%
\text{ } \\ 
N_{r}(B(z,R)\bigcap E)\geq C_{2}\left( \frac{R}{r}\right) ^{\alpha }\text{ }%
\forall z\in E%
\end{array}%
\right\} .
\end{equation*}
\end{definition}

The $\Phi $-dimensions were first thoroughly studied in \cite{GHM},
expanding upon the earlier work of \cite{FY}. The (quasi-) upper and lower
Assouad dimensions and the upper and lower $\theta $-Assouad spectrum are
again special cases, arising in the same manner as for measures. It is known
that 
\begin{equation*}
\dim _{L}E\leq \underline{\dim }_{\Phi }E\leq \underline{\dim }_{B}E\leq 
\overline{\dim }_{B}E\leq \overline{\dim }_{\Phi }E\leq \dim _{A}E
\end{equation*}%
and for closed sets $\dim _{L}E\leq \dim _{H}E$. For further background and
proofs, we refer the reader to the references mentioned in the introduction,
as well as Fraser's monograph, \cite{Fraserbook}.

Obviously, if $\Phi (x)\leq \Psi (x)$ for all $x>0$, then for any measure $%
\mu $ we have 
\begin{equation*}
\overline{\dim }_{\Psi }\mu \text{ }\leq \overline{\dim }_{\Phi }\mu \text{
and }\underline{\dim }_{\Phi }\mu \leq \underline{\dim }_{\Psi }\mu \text{ .}
\end{equation*}%
A similar statement holds for the $\Phi $-dimension of sets. The next
Proposition summarizes other relationships between these dimensions. For the
proofs of these facts and many other properties of the $\Phi $-dimensions of
measures, we refer the reader to \cite{HH} and the references cited there.

\begin{proposition}
{\ \label{basic}} Let $\Phi $ be a dimension function and $\mu $ be a
measure.

(i) Then $\overline{\dim }_{\Phi }\mu \geq \overline{\dim }_{\Phi }\ \mathrm{%
supp}\mu \geq \dim _{H}\mathrm{supp}\mu $ and%
\begin{equation}
\dim _{L}\mu \leq \underline{\dim }_{\Phi }\mu \leq \dim _{F}\mu \leq 
\overline{\dim }_{M}\mu \leq \overline{\dim }_{\Phi }\mu \leq \dim _{A}\mu .
\label{R1}
\end{equation}%
If $\mu $ is doubling, then $\underline{\dim }_{\Phi }\mu \leq \underline{%
\dim }_{\Phi }\mathrm{supp}\mu .$

(ii) If $\Phi (x)\rightarrow 0$ as $x\rightarrow 0,$ then $\underline{\dim }%
_{\Phi }\mu \leq \dim _{qL}\mu $ and $\dim _{qA}\mu \leq \overline{\dim }%
_{\Phi }\mu $ .

(iii) If there exists $x_{0}>0$ such that $\Phi (x)\leq C/\left\vert \log
x\right\vert $ for $0<x\leq x_{0},$ then $\overline{\dim }_{\Phi }\mu =\dim
_{A}\mu $ and $\underline{\dim }_{\Phi }\mu =\dim _{L}\mu $.

(iv) If $\Theta =\limsup_{x\rightarrow 0}\Phi (x)^{-1},$ then 
\begin{eqnarray*}
\underline{\dim }_{\Phi }\mu &\geq &\dim _{F}\mu -\Theta (\overline{\dim }%
_{M}\mu -\dim _{F}\mu ), \\
\overline{\dim }_{\Phi }\mu &\leq &\overline{\dim }_{M}\mu +\Theta (%
\overline{\dim }_{M}\mu -\dim _{F}\mu ).
\end{eqnarray*}
\end{proposition}

\subsection{The random set-up\label{setup}}

Let $I\subseteq \mathbb{R}^{D}$. We denote by $diam(I)$ the diameter of $I$.
Given $r>0$, we say the subset $J\subseteq I$ is an $r$\textbf{-similarity}
of $I$ if there is a similarity $S_{\sigma }$ such that $J=S_{\sigma }(I)$
and $diam(J)=r\cdot diam(I)$. (Thus $S_{\sigma }$ has contraction factor $r$%
.) We say the collection of $r$-similarities, $J_{1},...,J_{t}$, is $\tau $%
\textbf{-separated} if $d(J_{i},J_{j})\geq \tau \cdot r\cdot diam(I)$ for
all $i\neq j$. When such a collection of $t$ sets exists, we say $I$ has the 
$(t,r,\tau )$\textbf{-separation property}.

Of course, if $I$ contains a non-empty open set, and $\tau >0$ is given,
then $I$ will have the $(t,r,\tau )$-separation property for all $t\in 
\mathbb{N}$ and $r\leq r_{t},$ for some suitably small $r_{t}>0$. For
example, if $I_{0}=[0,1]\subseteq \mathbb{R}$, then $r_{t}=1/(t+\tau (t-1))$
will work. This condition can be viewed as a uniform strong separation
condition.

From here on, $I_{0}$ will denote a (fixed) compact subset of $\mathbb{R}%
^{D} $ with diameter $1$ and non-empty interior. We fix $0<\tau <1$ and for
each $t\in \mathbb{N}$ we choose $r_{t}\in (0,1/2]$ so that $I_{0}$ has $%
(t,r,\tau )$-separation property for all $r\leq r_{t}$. For each $t=2,3,...$
we define the probability simplices:%
\begin{equation*}
\mathcal{S}_{t}=\mathcal{\ }\{(x_{1},...,x_{t})\in \mathbb{R}%
^{t}:x_{i}>0,\sum_{i=1}^{t}x_{i}=1\},
\end{equation*}%
\begin{equation*}
\Omega _{t}=(0,r_{t}]\times \mathcal{S}_{t}\subseteq (0,1)\times \mathcal{S}%
_{t}\text{ }
\end{equation*}%
and 
\begin{equation*}
\text{ }\Omega _{0}=\bigcup_{t\geq 2}\Omega _{t}.
\end{equation*}%
Let $\pi $ be a Borel probability measure on $\Omega _{0}$ and let $\mathbb{P%
}$ be the product measure on the infinite product $\Omega =\Omega _{0}^{%
\mathbb{N}}$ induced by $\pi $.

Continuing, we define random variables $T$ and $r,$ and a random vector $p$
on $\Omega _{0}$. To do this, for $\omega \in \Omega _{0},$ we have $\omega
\in \Omega _{t}$ for some $t$ $=2,3,...$ with $\omega
=(r,p^{(1)},p^{(2)},...,p^{(t)})$. Using this, we define 
\begin{equation*}
T(\omega )=t\text{, }r(\omega )=r\text{, }p(\omega
)=(p^{(1)},p^{(2)},...,p^{(t)})\in \mathcal{S}_{t}\text{.}
\end{equation*}

A common way to define the measure $\pi $ on $\Omega _{0}$ is as a two-step
process where one first chooses the integer $t$ randomly according to some
distribution and then independently choose $r$ uniformly from $(0,r_{t}]$
and $p$ uniformly from $\mathcal{S}_{t}$.

With this framework, $\omega \in \Omega $ drawn according to $\mathbb{P}$
represents an independent and identically distributed random sample $%
(T_{n},r_{n},p_{n})$ from $\pi $ on $\Omega _{0}$.

Using this iid sample, we can now construct a random Moran fractal.
Beginning with the compact set $I_{0}$ and $\omega \in \Omega ,$ we select a
collection of $T_{1}(\omega )\;$subsets that are $r_{1}(\omega )$%
-similarities of $I_{0},$ $\{I_{0}^{(j)}\}_{j=1}^{T_{1}(\omega )}$, that are 
$\tau $-separated. This is possible as $r_{1}(\omega )\leq r_{T_{1}(\omega
)},$ so $I_{0}$ has the $(T_{1}(\omega ),r_{1}(\omega ),\tau )$-separation
property. We call the sets $I_{0}^{(j)},$ $j=1,...,T_{1}(\omega ),$ the
Moran sets of step (or level) 1. The Moran sets of step 1 have diameter $%
r_{1}$ and the distance between any two is at least $\tau r_{1}$.

Being similar to $I_{0},$ the sets $I_{0}^{(j)}$ also have the $(t,r,\tau )$%
-separation property for all $r\leq r_{t},$ so we may repeat this process.
Assume inductively that we have chosen the $\prod\limits_{j=1}^{n}T_{j}(%
\omega )$ Moran sets of step $n$. From each such set, $I_{n},$ we select $%
T_{n+1}(\omega )$ subsets that are $r_{n+1}(\omega )$-similarities of $I_{n}$
and are $\tau $-separated. These sets all have diameter $r_{1}\cdot \cdot
\cdot r_{n+1},$ are separated by a distance of at least $\tau r_{1}\cdot
\cdot \cdot r_{n+1}$ and are known as the \textbf{Moran sets of step }$n+1$.
The Moran sets of step $n+1$ that are subsets of a given Moran set $I_{n}$
of step $n$ are known as the children of the parent set $I_{n}$. We will use
the notation $I_{N}(x)$ for the unique Moran set of step $N$ containing $%
x\in E(\omega )$. Of course, $I_{N-1}(x)$ is its parent.

If we let $\mathcal{M}_{n}(\omega )$ be the union of the step $n$ Moran
sets, then the \textbf{random Moran set }$E(\omega )$ is the compact set 
\begin{equation*}
E(\omega )=\bigcap_{n=1}^{\infty }\mathcal{M}_{n}(\omega ).
\end{equation*}

We define a\textbf{\ random Moran measure} $\mu =\mu _{\omega }$ inductively
by the rule that if the $T_{n+1}$ children of the step $n$ Moran set $I$ are
labelled $I^{(1)},....,I^{(T_{n+1})},$ then $\mu (I^{(j)})=\mu
(I)p_{n+1}^{(j)}$ where $\mu (I_{0})=1$. The support of $\mu $ is the random
Moran set $E(\omega )$.

It is not difficult to see that the $\Phi $-dimensions of these random
measures is a tail event and hence our interest is in almost sure results.

\begin{example}
The classical middle-third Cantor set is a simple example of a Moran set
where $I_{0}=[0,1]$, $\tau =1/3$, $T_{n}(\omega )=2$ and $r_{n}(\omega )=1/3$
for all $\omega $ and $n$. The uniform Cantor measure is a special case of
this construction with probabilities $(1/2,1/2)$. More generally, any
self-similar set/measure arising from an iterated function system (IFS) $%
\{S_{j},p_{j}\}_{j=1}^{M},$ of equicontractive similarities $S_{j}$ acting
on $I_{0}$ and satisfying the strong separation condition, (meaning, the
sets $S_{j}(I_{0})$ are disjoint) and associated probabilities $p_{j},$ is a
random Moran set/measure in our sense.
\end{example}

\begin{example}
\label{1var}Another natural class of examples of random Moran sets are the
finite random 1-variable (homogeneous) models which arise from a finite
family of iterated function systems, $\{\mathcal{F}_{i}\}$, where each IFS$\ 
\mathcal{F}_{i}$ consists of equicontractive similarities acting on (the
common set) $I_{0}$ and satisfying the strong separation condition, where at
each step in the construction we randomly choose one IFS to apply at that
level. See \cite{Tr}.

More generally, our construction can be thought of as a random 1-variable
IFS construction where we have an uncountable family of IFSs with associated
probabilities.\qquad
\end{example}

Other examples are given in Section \ref{applications}.

\subsection{Preliminary Results}

Throughout the paper, we will assume $L$ is the smallest integer such that 
\begin{equation*}
2^{-L}\leq \tau /2.
\end{equation*}

First, we will verify that we can replace balls by suitable Moran sets for
the calculation of the $\Phi $-dimensions. The $\tau $-separation condition
is required for this.

\begin{lemma}
\label{BallsMoranInt}Fix $\omega \in \Omega $, $R>0$ and $x\in E(\omega )$.
Choose $N=N(\omega )$ such that 
\begin{equation*}
r_{1}(\omega )\cdot \cdot \cdot r_{N+1}(\omega )\leq R<r_{1}(\omega )\cdot
\cdot \cdot r_{N}(\omega ).
\end{equation*}%
Then 
\begin{equation*}
I_{N+1}(x)\cap E(\omega )\subseteq B(x,R)\cap E(\omega )\subseteq I_{N-L}(x)
\end{equation*}%
and hence 
\begin{equation*}
\mu (I_{N+1}(x))\leq \mu (B(x,R))\leq \mu (I_{N-L}(x)).
\end{equation*}
\end{lemma}

\begin{proof}
Since $R$ is at least the diameter of any Moran set of step $N+1$, $%
I_{N+1}(x)\subseteq B(x,R)$.

Now consider all the step $N$ Moran sets that intersect $B(x,R)$. If two of
these sets, say $I^{(1)},I^{(2)},$ are subsets of different Moran sets of
level $N-k$, then the distance between these two sets, call it $\delta ,$ is
at least the minimum distance between any two level $N-k$ Moran sets and
hence is at least $\tau r_{1}\cdot \cdot \cdot r_{N-k}$. But $\tau \geq
2^{1-L}$ and $r_{j}\leq 1/2,$ hence%
\begin{equation*}
\delta \geq 2\cdot 2^{-L}r_{1}\cdot \cdot \cdot r_{N-k}\geq 2r_{1}\cdot
\cdot \cdot r_{N-k}r_{N-k+1}\cdot \cdot \cdot r_{N-k+L}\text{.}
\end{equation*}%
As $I^{(1)}$ and $I^{(2)}$ both intersect $B(x,R),$ we must have $\delta
\leq 2R<2r_{1}\cdot \cdot \cdot r_{N}$ and therefore $k<L$. Thus all the
step $N$ Moran sets that intersect $B(x,R)$ are subsets of the same Moran
set of level $N-L$, namely $I_{N-L}(x)$.
\end{proof}

Next, we introduce the positive-valued random variables 
\begin{equation*}
M_{n}=\max_{j}p_{n}^{(j)},\text{ }m_{n}=\min_{j}p_{n}^{(j)},
\end{equation*}%
\begin{equation*}
X_{n}=-\log M_{n},\text{ }Y_{n}=-\log m_{n},\text{ and }Z_{n}=-\log r_{n}.
\end{equation*}%
Of course, the collections $(M_{n})_{n},$ $(m_{n})_{n},$ $(X_{n})_{n},$ $%
(Y_{n})_{n}$ and $(Z_{n})$ are all independent and identically distributed.

Since $\sum_{j=1}^{T_{n}}p_{n}^{(j)}$ $=1$ for each $n$, $M_{n}\geq 1/T_{n}$%
, so if the number of children is bounded, then $\mathbb{E}(e^{\lambda
X_{1}})<\infty $ for any $\lambda $. More generally, $\mathbb{E(}e^{\lambda
W_{1}})<\infty $ for $W_{1}$ any of $X_{1},Y_{1}$ or $Z_{1}$ if there exists 
$\delta >0$ such that $W_{1}(\omega )>\delta $ a.s.

\medskip

\textbf{Basic Assumption}: Throughout this paper we will assume that there
exists some $A>0$ such that $\mathbb{E}(e^{\lambda Z_{1}})<\infty $ for all $%
|\lambda |\leq A$.

\bigskip

Note that this happens if and only if ${\mathbb{E}}(r_{1}^{-A})<\infty ,$ as
is true, for example, when the ratios $r_{n}$ are uniformly distributed.
This is an important assumption because we will make heavy use of the
following probabilistic result, sometimes known as the Chernoff technique.

\begin{theorem}
\label{prob}(\cite[Thm. 2.6]{Pe}) Suppose $(F_{n})$ are iid rv's and for
some $A>0$, $\mathbb{E}(e^{\lambda F_{1}})<\infty $ for all $|\lambda |\leq
A $. Then for all $a>0,$ there exists $b>0$ such that 
\begin{equation*}
\mathbb{P}\left( \left\vert \sum_{j=1}^{k}F_{j}-kE(F_{1})\right\vert \geq
ak\right) \leq \exp (-bk)
\end{equation*}%
for all $k\in \mathbb{N}$.
\end{theorem}

\begin{remark}
(i) If $W$ is a non-negative random variable with $\mathbb{P}(W\leq x)\leq
Cx^{\theta }$ for small $x$ and $\theta >0$, then ${\mathbb{E}}(e^{-\lambda
\log W})=\mathbb{E(}W^{-\lambda })<\infty $ for small $|\lambda |$. In
particular this holds if $W$ has a probability density function $f(x)$ with $%
f(x)\leq Cx^{q}$ and $q>-1$.

(ii) Since $\{p\in {\mathcal{S}}_{t}:\min_{i}p_{i}\leq z\}\subseteq
(\partial {\mathcal{S}}_{t})_{2z}$ (the $2z$-dilation of the boundary of ${%
\mathcal{S}}_{t}$), part (i) means it is enough that $\mathbb{P}((\partial {%
\mathcal{S}}_{t})_{z})\leq Cz^{\theta }$ for some $C,\theta >0$ and small $z$
to have ${\mathbb{E}}(e^{\lambda Y})<\infty $ for small $\left\vert \lambda
\right\vert $. Thus if the probability distribution on ${\mathcal{S}}_{t}$
has bounded density function, this will be true.

(iii) Since $\max_{i}p_{i}\geq 1/t$ for $p\in {\mathcal{S}}_{t}$, we know
that $M_{n}$ is close to zero only when $T_{n}$, the number of children, is
very large. Very roughly, $\mathbb{P}(M_{n}\leq \lambda )\leq \mathbb{P}%
(T_{n}\geq 1/\lambda )$. If ${\mathbb{P}}(T=t)\leq Ct^{-\theta }$ with $%
\theta >1$ and the distribution on ${\mathcal{S}}_{t}$ is uniform, then ${%
\mathbb{E}}(e^{\lambda X})<\infty $ for small $\left\vert \lambda
\right\vert $. Thus this can happen even if ${\mathbb{E}}(T)=\infty $ (take
any $\theta \in (1,2)$).
\end{remark}

\begin{notation}
Given a dimension function $\Phi $ and random Moran set $E(\omega ),$ we
define the associated \textbf{depth function} $\phi =\phi _{\omega }:\mathbb{%
N\rightarrow N}$ by the rule that $\phi _{\omega }(n)$ is the minimal
positive integer $k$ such that 
\begin{equation*}
r_{1}(\omega )\cdot \cdot \cdot r_{n+k}(\omega )\leq \left( r_{1}\cdot \cdot
\cdot r_{n}\right) ^{1+\Phi (r_{1}\cdot \cdot \cdot r_{n})}.
\end{equation*}
\end{notation}

Consequently, 
\begin{eqnarray}
r_{1}\cdot \cdot \cdot r_{n+\phi (n)} &\leq &(r_{1}\cdot \cdot \cdot
r_{n})^{1+\Phi (r_{1}\cdot \cdot \cdot r_{n})}\text{ and}  \label{rphi} \\
r_{1}\cdot \cdot \cdot r_{n+\phi (n)-1} &>&(r_{1}\cdot \cdot \cdot
r_{n})^{1+\Phi (r_{1}\cdot \cdot \cdot r_{n})}.  \label{rphiupper}
\end{eqnarray}%
This notion was introduced in \cite{GHM} to study the formulas for the $\Phi 
$-dimensions of (deterministic) Cantor sets. There it was shown that if $C$
is the central Cantor set with intervals of length $r_{1}\cdot \cdot \cdot
r_{n}$ at step $n$ and $\inf_{n}r_{n}>0$, then 
\begin{eqnarray*}
\overline{\dim }_{\Phi }C &=&\limsup_{n}\left( \sup_{k\geq \phi (n)}\frac{%
n\log 2}{\log r_{n+1}\cdot \cdot \cdot r_{n+k}}\right) , \\
\underline{\dim }_{\Phi }C &=&\liminf_{n}\left( \inf_{k\geq \phi (n)}\frac{%
n\log 2}{\log r_{n+1}\cdot \cdot \cdot r_{n+k}}\right) .
\end{eqnarray*}%
In this situation, there is a simple relationship between $\phi $ and $\Phi $%
: with $C=\inf r_{n}/\sup r_{n}$ we have $\phi (n)-1\leq Cn\Phi (r_{1}\cdot
\cdot \cdot r_{n})\leq \phi (n)$.

For the random problem, it will be helpful to have information about the
size of $\phi _{\omega }(n)$ that is independent of $\omega $ (in an almost
sure sense). As we will see in the next result, the answer depends on how $%
\Phi (x)$ compares with the function $\log |\log x|/\left\vert \log
x\right\vert $.

\begin{notation}
Given functions $G,H:(0,1)\rightarrow \mathbb{R}^{+},$ we define%
\begin{equation}
\zeta _{N}^{(G)}=\frac{G(2^{-N})\log (N\log 2)}{2\mathbb{E}\mathbf{(}Z_{1})}
\label{zeta}
\end{equation}%
and 
\begin{equation}
\chi _{N}^{(H)}=\frac{H(2^{-N})\log (2N\mathbb{E}\mathbf{(}Z_{1}))}{\log 2}%
\text{.}  \label{chi}
\end{equation}
\end{notation}

Apply Theorem \ref{prob} to choose a constant $B$ so that for all $k\in 
\mathbb{N}$,%
\begin{equation}
\mathbb{P}\left( \left\vert \sum_{j=1}^{k}Z_{j}-k\mathbb{E}%
(Z_{1})\right\vert \geq k\mathbb{E}(Z_{1})\right) \leq \exp (-Bk).
\label{PetrovB}
\end{equation}

\begin{lemma}
\label{phi}(i) Suppose $\Phi (x)\geq G(x)\log \left\vert \log x\right\vert
/\left\vert \log x\right\vert $ where $G$ is non-decreasing as $x$ decreases
to $0$ and $G(x)\geq 4\mathbb{E}(Z_{1})/B$ for all $x\in (0,1)$. Then 
\begin{equation*}
\mathbb{P}(\omega :\phi _{\omega }(N)<\zeta _{N}^{(G)}\text{ }i.o.)=0.
\end{equation*}

(ii) If $\Phi (x)\leq H(x)\log \left\vert \log x\right\vert /\left\vert \log
x\right\vert $ where $H(x)$ is non-increasing as $x$ decreases to $0,$ then $%
\mathbb{P}(\omega :\phi _{\omega }(N)>\chi _{N}^{(H)}$ $i.o.)=0.$
\end{lemma}

\begin{proof}
(i) If $\phi _{\omega }(N)<\zeta _{N}^{(G)}=\zeta _{N}$, then 
\begin{equation*}
r_{1}r_{2}\cdot \cdot \cdot r_{N+\zeta _{N}}\leq r_{1}r_{2}\cdot \cdot \cdot
r_{N+\phi _{\omega }(N)}\leq (r_{1}\cdot \cdot \cdot r_{N})^{1+\Phi
(r_{1}\cdot \cdot \cdot r_{N})},
\end{equation*}%
so%
\begin{equation*}
r_{N+1}r_{N+2}\cdot \cdot \cdot r_{N+\zeta _{N}}\leq (r_{1}\cdot \cdot \cdot
r_{N})^{\Phi (r_{1}\cdot \cdot \cdot r_{N})}
\end{equation*}%
and therefore 
\begin{equation*}
\sum_{j=N+1}^{N+\zeta _{N}}-\log r_{j}\geq \Phi (r_{1}\cdot \cdot \cdot
r_{N})\left\vert \log r_{1}\cdot \cdot \cdot r_{N}\right\vert \geq
G(r_{1}\cdot \cdot \cdot r_{N})\log \left\vert \log r_{1}\cdot \cdot \cdot
r_{N}\right\vert .
\end{equation*}%
But $r_{1}\cdot \cdot \cdot r_{N}\leq 2^{-N},$ so by monotonicity, $%
G(r_{1}\cdot \cdot \cdot r_{N})\geq G(2^{-N})$ and $\left\vert \log
r_{1}\cdot \cdot \cdot r_{N}\right\vert \geq N\log 2$. Hence if $\phi
_{\omega }(N)<\zeta _{N},$ then 
\begin{equation*}
\sum_{j=N+1}^{N+\zeta _{N}}Z_{j}=\sum_{j=N+1}^{N+\zeta _{N}}-\log r_{j}\geq
G(2^{-N})\log (N\log 2)=2\zeta _{N}\mathbb{E}(Z_{1}),
\end{equation*}%
so the choice of $B$ gives 
\begin{align*}
\mathbb{P}(\omega :\phi _{\omega }(N)<\zeta _{N})& =\mathbb{P}\left(
\sum_{j=1}^{\zeta _{N}}Z_{j}\geq 2\zeta _{N}\mathbb{E}(Z_{1})\right) \\
& \leq \mathbb{P}\left( \left\vert \sum_{j=1}^{\zeta _{N}}Z_{j}-\zeta _{N}%
\mathbb{E}(Z_{1})\right\vert \geq \zeta _{N}\mathbb{E}(Z_{1})\right) \leq
\exp (-B\zeta _{N}).
\end{align*}%
Since $G(2^{-N})\geq $ $4\mathbb{E}(Z_{1})/B$, there is a constant $c$ such
that 
\begin{equation}
\exp (-B\zeta _{N})=(N\log 2)^{-BG(2^{-N})/2\mathbb{E}(Z_{1})}\leq cN^{-2}.
\label{summing}
\end{equation}%
Thus 
\begin{equation*}
\sum_{N}\mathbb{P}(\omega :\phi _{\omega }(N)<\zeta _{N})\leq
c\sum_{N}N^{-2}<\infty
\end{equation*}%
and the Borel Cantelli lemma implies that $\mathbb{P}(\omega :\phi _{\omega
}(N)<\zeta _{N}^{(G)})$ i.o.$)=0$.

(ii) The arguments are similar, but easier, for (ii). If $\phi _{\omega
}(N)>\chi _{N}^{(H)}$, then 
\begin{equation*}
r_{1}\cdot \cdot \cdot r_{N+\chi _{N}}\geq r_{1}\cdot \cdot \cdot r_{N+\phi
(N)-1}\geq (r_{1}\cdot \cdot \cdot r_{N})^{1+\Phi (r_{1}\cdot \cdot \cdot
r_{N})}.
\end{equation*}%
Hence 
\begin{eqnarray*}
\chi _{N}^{(H)}\log 2 &\leq &-\log r_{N+1}\cdot \cdot \cdot r_{N+\chi
_{N}}\leq -\Phi (r_{1}\cdot \cdot \cdot r_{N})\log (r_{1}\cdot \cdot \cdot
r_{N}) \\
&=&H(r_{1}\cdot \cdot \cdot r_{N})\log \left\vert \log r_{1}\cdot \cdot
\cdot r_{N}\right\vert \leq H(2^{-N})\log \left\vert \log r_{1}\cdot \cdot
\cdot r_{N}\right\vert .
\end{eqnarray*}%
Putting in the formula for $\chi _{N}^{(H)}$ and taking exponentials gives%
\begin{equation*}
2N\mathbb{E}(Z_{1})\leq \left\vert \log r_{1}\cdot \cdot \cdot
r_{N}\right\vert =\sum_{i=1}^{N}-\log r_{i}.
\end{equation*}%
Applying Theorem \ref{prob} again, we obtain 
\begin{equation*}
\mathbb{P(}\phi _{\omega }(N)>\chi _{N}^{(H)})\leq \mathbb{P}\left(
\sum_{i=1}^{N}Z_{i}\geq 2N\mathbb{E}(Z_{1})\right) \leq \exp (-BN)
\end{equation*}%
and the Borel Cantelli lemma gives the result.
\end{proof}

\section{Dimension Results for Large $\Phi $\label{largesection}}

\textbf{Terminology}: We call a dimension function $\Phi $ \textbf{large} if
it satisfies $\Phi (x)\gg \log |\log x|/|\log x|$.

The constant functions $\Phi =\delta >0$ or any dimension function $\Phi
\rightarrow \infty $ are examples of large dimension functions.

\begin{theorem}
\label{MainLarge} There is a set $\Gamma ,$ of full measure in $\Omega ,$
with the following properties:

(i) If there exists $A>0$ such that $\mathbb{E(}e^{\lambda X_{1}})<\infty $
for all $\left\vert \lambda \right\vert \leq A$, then \underline{$\dim $}$%
_{\Phi }\mu _{\omega }=\mathbb{E}(X_{1})/\mathbb{E(}Z_{1})$ for all large
dimension functions $\Phi $ and for all $\omega \in \Gamma $.

\smallskip

(ii) If there exists $A>0$ such that $\mathbb{E(}e^{\lambda Y_{1}})<\infty $
for all $\left\vert \lambda \right\vert \leq A$, then $\overline{\dim }%
_{\Phi }\mu _{\omega }=\mathbb{E}(Y_{1})/\mathbb{E(}Z_{1})$ for all large
dimension functions $\Phi $ and for all $\omega \in \Gamma $.
\end{theorem}

Since positive constant functions are large dimension functions, the
following corollary is immediate.

\begin{corollary}
If there exists $A>0$ such that $\mathbb{E(}e^{\lambda X_{1}})$, $\mathbb{E(}%
e^{\lambda Y_{1}})<\infty $ for all $\left\vert \lambda \right\vert \leq A$,
then $\dim _{qL}\mu _{\omega }=\dim _{F}\mu _{\omega }=$ $\mathbb{E}(X_{1})/%
\mathbb{E(}Z_{1})$ and $\dim _{qA}\mu _{\omega }=\overline{\dim }_{M}\mu
_{\omega }=\mathbb{E}(Y_{1})/\mathbb{E(}Z_{1})$ almost surely.
\end{corollary}

We will abbreviate the statement `there exists $A>0$ such that $\mathbb{E(}%
e^{\lambda W})<\infty $ for all $\left\vert \lambda \right\vert \leq A\,$'
by `$\mathbb{E(}e^{\lambda W})<\infty $ for small $\left\vert \lambda
\right\vert $' and set 
\begin{equation*}
\underline{d}=\mathbb{E}(X_{1})/\mathbb{E(}Z_{1})\text{, \quad }\overline{d}=%
\mathbb{E}(Y_{1})/\mathbb{E(}Z_{1}).\ 
\end{equation*}%
We begin with two lemmas. We recall that $\zeta _{N}^{(\cdot )}$ was defined
in (\ref{zeta}).

\begin{lemma}
\label{L1}(i) Assume $\mathbb{E(}e^{\lambda X_{1}})<\infty $ for small $%
\left\vert \lambda \right\vert $. For each $\varepsilon >0,$ there is a
constant $K_{\varepsilon }\geq 1$ such that if 
\begin{equation*}
F_{N}(\varepsilon )=\{\omega :\exists k\geq \zeta _{N}^{(K_{\varepsilon })}-1%
\text{ with }\prod\limits_{i=N+2}^{N+k-L-1}M_{i}{}^{-1}<\prod%
\limits_{i=N+1}^{N+k}r_{i}{}^{-\underline{d}(1-\varepsilon )}\},
\end{equation*}%
then $\mathbb{P}(F_{N}(\varepsilon )$ i.o.$)=0$.

(ii) Assume $\mathbb{E(}e^{\lambda Y_{1}})<\infty $ for small $\left\vert
\lambda \right\vert $. For each $\varepsilon >0,$ there is a constant $%
K_{\varepsilon }\geq 1$ such that if 
\begin{equation*}
G_{N}(\varepsilon )=\{\omega :\exists k\geq \zeta _{N}^{(K_{\varepsilon })}-1%
\text{ with }\prod\limits_{i=N-L+1}^{N+k}m_{i}{}^{-1}>\prod%
\limits_{i=N+2}^{N+k-1}r_{i}{}^{-\overline{d}(1+\varepsilon )}\},
\end{equation*}%
then $\mathbb{P}(G_{N}(\varepsilon )$ i.o.$)=0$.
\end{lemma}

\begin{remark}
We remark we can take the same constant $K_{\varepsilon }$ in both parts;
the choice of constant will be clear from the proof. We also note that once $%
N$ is suitably large, we will have $\zeta _{N}^{(K_{\varepsilon })}>\max
(3,L+3),$ whence the products above are well defined.
\end{remark}

\begin{proof}
To simplify notation, we will let $\zeta _{N}^{(K_{\varepsilon })}=\zeta
_{N}^{\varepsilon }$.

(i) Upon taking logarithms, we have 
\begin{equation*}
F_{N}(\varepsilon )=\{\omega :\exists k\geq \zeta _{N}^{\varepsilon }-1\text{
with}\sum_{j=N+2}^{N+k-L-1}X_{j}<\underline{d}(1-\varepsilon
)\sum_{j=N+1}^{N+k}Z_{j}\text{ }\}.
\end{equation*}%
The definition of \underline{$d$} ensures that 
\begin{equation*}
\underline{d}(1-\varepsilon )\sum_{j=N+1}^{N+k}Z_{j}=\underline{d}%
(1-\varepsilon )\left[ \sum_{j=N+1}^{N+k}Z_{j}-k\mathbb{E(}Z_{1})\right]
+(1-\varepsilon )k\mathbb{E(}X_{1}).
\end{equation*}%
Pick $N_{0}$ such that if $k\geq $ $\zeta _{N_{0}}^{\varepsilon }-1,$ then $%
(1-\varepsilon )k\leq (1-\varepsilon /2)(k-L-2)$ and $(k-L-2)\varepsilon /4$ 
$\geq k\varepsilon /8$. For all $N\geq N_{0}$, the set $F_{N}(\varepsilon )$
is contained in%
\begin{equation*}
\bigcup_{k\geq \zeta _{N}^{\varepsilon }-1}\left\{ \text{ }%
\sum_{j=N+2}^{N+k-L-1}X_{j}<\underline{d}(1-\varepsilon
)[\sum_{j=N+1}^{N+k}Z_{j}-k\mathbb{E(}Z_{1})]+(1-\frac{\varepsilon }{2}%
)(k-L-2)\mathbb{E(}X_{1})\right\}
\end{equation*}%
\begin{eqnarray*}
&=&\bigcup_{k\geq \zeta _{N}^{\varepsilon }-1}\left\{ 
\begin{array}{c}
\text{ }\left[ \sum_{j=N+2}^{N+k-L-1}X_{j}-(k-L-2)\mathbb{E(}X_{1})\right] +%
\underline{d}(1-\varepsilon )\left[ k\mathbb{E(}Z_{1})-%
\sum_{j=N+1}^{N+k}Z_{j}\right] \\ 
<-\frac{(k-L-2)\varepsilon }{2}\mathbb{E(}X_{1})%
\end{array}%
\right\} \\
&\subseteq &\bigcup_{k\geq \zeta _{N}^{\varepsilon }-1}\left\{ 
\begin{array}{c}
\text{ }\left[ \sum_{j=N+2}^{N+k-L-1}X_{j}-(k-L-2)\mathbb{E(}X_{1})\right] +%
\underline{d}(1-\varepsilon )\left[ k\mathbb{E(}Z_{1})-%
\sum_{j=N+1}^{N+k}Z_{j}\right] \\ 
<-\frac{(k-L-2)\varepsilon }{4}\mathbb{E(}X_{1})-\frac{k\varepsilon }{8}%
\mathbb{E(}X_{1})%
\end{array}%
\right\} .
\end{eqnarray*}%
Hence $F_{N}(\varepsilon )\subseteq $ $\bigcup_{k\geq \zeta
_{N}^{\varepsilon }-1}(A_{k}\bigcup B_{k})$ where 
\begin{equation*}
A_{k}=\left\{ \omega :\left\vert \sum_{i=2}^{k-L-1}X_{i}-(k-L-2)\mathbb{E(}%
X_{1})\right\vert >\frac{(k-L-2)\varepsilon \mathbb{E(}X_{1})}{4}\right\}
\end{equation*}%
and%
\begin{equation*}
B_{k}=\left\{ \omega :\left\vert k\mathbb{E(}Z_{1})-\sum_{i=1}^{k}Z_{i}%
\right\vert >\frac{k\varepsilon \mathbb{E(}X_{1})}{8\underline{d}%
(1-\varepsilon )}\right\} =\left\{ \omega :\left\vert k\mathbb{E(}%
Z_{1})-\sum_{i=1}^{k}Z_{i}\right\vert >\frac{k\varepsilon \mathbb{E(}Z_{1})}{%
8(1-\varepsilon )}\right\} .
\end{equation*}

From Theorem \ref{prob}, there are constants $a_{\varepsilon
},b_{\varepsilon }>0$ such that $\mathbb{P(}A_{k})\leq e^{-a_{\varepsilon
}(k-L-2)}$ and $\mathbb{P(}B_{k})\leq e^{-b_{\varepsilon }k}$ for all $k$.
Thus there are constants $C_{\varepsilon },c_{\varepsilon }>0$ such that for
large enough $N,$%
\begin{equation*}
\mathbb{P}(F_{N}(\varepsilon ))\leq \sum_{k\geq \zeta _{N}-1}\mathbb{P(}%
A_{k})+\mathbb{P(}B_{k})\leq C_{\varepsilon }e^{-c_{\varepsilon }\zeta
_{N}^{\varepsilon }}.
\end{equation*}%
Choose $K_{\varepsilon }$ sufficiently large to ensure that $\exp
(-c_{\varepsilon }\zeta _{N}^{\varepsilon })\leq N^{-2}$ for all $N$. With
this choice, $\sum_{N=1}^{\infty }\mathbb{P}(F_{N}(\varepsilon ))<\infty $
and hence the Borel Cantelli lemma gives the desired result.

(ii) This follows in a very similar fashion. The details are left for the
reader.
\end{proof}

\bigskip

\begin{lemma}
\label{L2}(i) Assume $\mathbb{E(}e^{\lambda X_{1}})<\infty $ for small $%
\left\vert \lambda \right\vert $. For each $\varepsilon >0,$ there is a
constant $K_{\varepsilon }^{\prime }\geq 1$ such that $\mathbb{P}%
(F_{N}^{\prime }(\varepsilon )$ i.o.$)=0$ where%
\begin{equation*}
F_{N}^{\prime }(\varepsilon )=\{\omega :\exists k\geq \zeta
_{N}^{(K_{\varepsilon }^{\prime })}\text{ with }%
\prod_{i=N+1}^{N+k}M_{i}^{-1}>\prod_{i=N+1}^{N+k}r_{i}{}^{-\underline{d}%
(1+\varepsilon )}\}.
\end{equation*}

(ii) Assume $\mathbb{E(}e^{\lambda Y_{1}})<\infty $ for small $\left\vert
\lambda \right\vert $. For each $\varepsilon >0,$ there is a constant $%
K_{\varepsilon }^{\prime }\geq 1$ such that $\mathbb{P}(G_{N}^{\prime
}(\varepsilon )$ i.o.$)=0$ where 
\begin{equation*}
G_{N}^{\prime }(\varepsilon )=\{\omega :\exists k\geq \zeta
_{N}^{(K_{\varepsilon }^{\prime })}\text{ with }%
\prod_{i=N+1}^{N+k}m_{i}^{-1}<\prod_{i=N+1}^{N+k}r_{i}{}^{-\overline{d}%
(1+\varepsilon )}\}.
\end{equation*}
\end{lemma}

\begin{proof}
Again we will only prove (i) as (ii) is similar and to simplify notation we
will write $\zeta _{N}^{(K_{\varepsilon }^{\prime })}=\zeta
_{N}^{\varepsilon }$.

It is straightforward to see that 
\begin{equation*}
F_{N}^{\prime }(\varepsilon )=\bigcup_{k\geq \zeta _{N}^{\varepsilon
}}\left\{ \text{ }\sum_{j=N+1}^{N+k}X_{j}>\underline{d}(1+\varepsilon
)\sum_{j=N+1}^{N+k}Z_{j}\right\}
\end{equation*}%
\begin{eqnarray*}
&\subseteq &\bigcup_{k\geq \zeta _{N}^{\varepsilon }}\left\{
[\sum_{j=N+1}^{N+k}X_{j}-\mathbb{E}(X_{1})k]+\underline{d}(1+\varepsilon )[k%
\mathbb{E}(Z_{1})-\sum_{j=N+1}^{N+k}Z_{j}]>\varepsilon k\mathbb{E}%
(X_{1})\right\} \\
&\subseteq &\bigcup_{k\geq \zeta _{N}^{\varepsilon }}\left\{ \left\vert
\sum_{j=1}^{k}X_{j}-\mathbb{E}(X_{1})k\right\vert \geq \frac{\varepsilon k%
\mathbb{E}(X_{1})}{2}\right\} \bigcup \left\{ \left\vert k\mathbb{E}%
(Z_{1})-\sum_{j=1}^{k}Z_{j}\right\vert \geq \frac{\varepsilon k\mathbb{E}%
(Z_{1})}{2(\varepsilon +1)}\right\} .
\end{eqnarray*}

From Petrov's theorem, $\mathbb{P}(F_{N}^{\prime }(\varepsilon ))\leq
C_{\varepsilon }e^{-c_{\varepsilon }\zeta _{N}^{\varepsilon }}$ for suitable
constants $C_{\varepsilon },c_{\varepsilon }>0.$ Again, make the choice of $%
K_{\varepsilon }^{\prime }$ to ensure $\sum_{N}$ $\mathbb{P}(F_{N}^{\prime
}(\varepsilon ))<\infty $.

\medskip
\end{proof}

\begin{proof}
\lbrack of Theorem \ref{MainLarge}] We begin with some initial observations
and notation that will be relevant to both the lower and upper $\Phi $%
-dimensions.

\smallskip

Let $\varepsilon >0$ and take $\mathcal{K}_{\varepsilon }=\max
(K_{\varepsilon },4\mathbb{E}(Z_{1})/B)$ where $K_{\varepsilon }$ is the
constant from Lemma \ref{L1} and $B$ arises from the probability theorem as
outlined in (\ref{PetrovB}). Let 
\begin{equation*}
\Phi ^{(\mathcal{K}_{\varepsilon })}(x)=\mathcal{K}_{\varepsilon }\log
\left\vert \log x\right\vert /\left\vert \log x\right\vert
\end{equation*}
and take $\zeta _{N}^{(\mathcal{K}_{\varepsilon })}$ as defined in (\ref%
{zeta}). To simplify notation, we will write $\Phi ^{\varepsilon },\phi
^{\varepsilon }$ and $\zeta _{N}^{\varepsilon }$.

According to Lemma \ref{phi}, there is a set, $\Gamma _{\varepsilon },$ of
full measure, with the property that for every $\omega \in \Gamma
_{\varepsilon }$ there is some integer $\mathcal{N}_{\omega }$ such that for
all $N\geq \mathcal{N}_{\omega }$ we have $\phi _{\omega }^{(\mathcal{%
\varepsilon })}(N)\geq $ $\zeta _{N}^{(\varepsilon )}$. Let $\Gamma
_{\varepsilon }(j)$ be the subset of $\Gamma _{\varepsilon }$ with $\mathcal{%
N}_{\omega }=j\in \mathbb{N}$.

Since $\zeta _{N}^{\varepsilon }\rightarrow \infty ,$ we can choose $J_{0}$
such that if $N\geq J_{0},$ then $\zeta _{N}^{\varepsilon }\geq L+4$. Take $%
\omega \in \Gamma _{\varepsilon }(j)$ and put $J_{j}=\max (j,J_{0})$.

Given $0<r<R<r_{1}\cdot \cdot \cdot r_{J_{j}}$, choose $N=N(\omega )$ and $%
n=n(\omega )>N$ such that 
\begin{eqnarray}
r_{1}\cdot \cdot \cdot r_{N+1} &\leq &R<r_{1}\cdot \cdot \cdot r_{N}\text{
and}  \label{sizerR} \\
r_{1}\cdot \cdot \cdot r_{n} &\leq &r<r_{1}\cdot \cdot \cdot r_{n-1}.  \notag
\end{eqnarray}%
Note that $N\geq J_{j}$ . Furthermore, we have the bounds%
\begin{equation}
(r_{N+2}\cdot \cdot \cdot r_{n-1})^{-1}\leq \frac{R}{r}\leq (r_{N+1}\cdot
\cdot \cdot r_{n})^{-1}.  \label{Rr}
\end{equation}%
If $r\leq r_{1}\cdot \cdot \cdot r_{N+1+\phi ^{\varepsilon }(N+1)}$, then
since the function $x^{1+\Phi (x)}$ is increasing, (\ref{rphi}) tells us that%
\begin{equation*}
r\leq (r_{1}\cdot \cdot \cdot r_{N+1})^{1+\Phi ^{\varepsilon }(r_{1}\cdot
\cdot \cdot r_{N+1})}\leq R^{1+\Phi ^{\varepsilon }(R)}.
\end{equation*}%
On the other hand, if $r\leq R^{1+\Phi ^{\varepsilon }(R)}$, then by (\ref%
{rphiupper}),%
\begin{equation*}
r\leq (r_{1}\cdot \cdot \cdot r_{N})^{1+\Phi ^{\varepsilon }(r_{1}\cdot
\cdot \cdot r_{N})}\leq r_{1}\cdot \cdot \cdot r_{N+\phi ^{\varepsilon
}(N)-1}
\end{equation*}%
and hence $n\geq N+\phi _{\omega }^{\varepsilon }(N)-1\geq N+\zeta
_{N}^{\varepsilon }-1$, so that $n-L-1\geq N+2$.

From Lemma \ref{BallsMoranInt} we know 
\begin{equation*}
\frac{\mu (I_{N+1}(x))}{\mu (I_{n-L-1}(x))}\leq \frac{\mu (B(x,R))}{\mu
(B(x,r))}\leq \frac{\mu (I_{N-L}(x))}{\mu (I_{n}(x))}.
\end{equation*}%
Since 
\begin{equation}
\frac{\mu (I_{j}(x))}{\mu (I_{j+s}(x))}=\frac{\mu (I_{j}(x))}{\mu
(I_{j}(x))p_{j+1}^{(i_{j+1})}\cdot \cdot \cdot p_{j+s}^{(i_{j+s})}}
\end{equation}%
for a suitable choice of probabilities $p_{\ell }^{(i_{\ell })}$, $\ell
=j+1, $ $...,j+s$, we have 
\begin{equation}
(M_{N+2}\cdot \cdot \cdot M_{n-L-1})^{-1}\leq \frac{\mu (B(x,R))}{\mu
(B(x,r))}\leq (m_{N-L+1}\cdot \cdot \cdot m_{n})^{-1}.  \label{ratiobounds}
\end{equation}

\medskip

\medskip

\textbf{Computation of the almost sure lower }$\Phi $\textbf{-dimension: }%
Assume $\mathbb{E(}e^{\lambda X_{1}})<\infty $ for small $\left\vert \lambda
\right\vert $.

Since $\phi _{\omega }^{(\varepsilon )}(N)\geq $ $\zeta _{N}^{\varepsilon }$
for $N\geq J_{j},$ with the notation $F_{N}(\varepsilon )$ from Lemma \ref%
{L1}(i) we have 
\begin{equation*}
\left\{ \omega \in \Gamma _{\varepsilon }(j):\exists k\geq \phi _{\omega
}^{\varepsilon }(N)-1\text{ with }\prod\limits_{i=N+2}^{N+k-L-1}M_{i}^{-1}<%
\prod\limits_{i=N+1}^{N+k}r_{i}{}^{-\underline{d}(1-\varepsilon )}\right\}
\subseteq F_{N}(\varepsilon ).
\end{equation*}%
By Lemma \ref{L1}, $\mathbb{P}(F_{N}(\varepsilon )$ i.o$)=0$, so $%
\{(F_{N}(\varepsilon )$ i.o$\}^{c}=\Lambda _{\varepsilon }$ where $\Lambda
_{\varepsilon }$ has full measure. Thus the set $\Delta _{\varepsilon
}(j)=\Lambda _{\varepsilon }\bigcap \Gamma _{\varepsilon }(j)$ has full
measure in $\Gamma _{\varepsilon }(j)$. Moreover, it has the property that
for each $\omega \in \Delta _{\varepsilon }(j),$ there is an integer $%
N_{\varepsilon ,j}(\omega )\geq J_{j}$ so that if $N\geq N_{\varepsilon
,j}(\omega ),$ then for every $k\geq \phi _{\omega }^{\varepsilon }(N)-1$ we
have 
\begin{equation*}
(M_{N+2}\cdot \cdot \cdot M_{N+k-L-1})^{-1}\geq (r_{N+1}\cdot \cdot \cdot
r_{N+k})^{-\underline{d}(1-\varepsilon )}.
\end{equation*}

For $\omega \in \Delta _{\varepsilon }(j)$, set $\rho (\varepsilon ,\omega
)=r_{1}\cdot \cdot \cdot r_{N_{\varepsilon ,j}(\omega )}$. If $R\leq \rho
(\varepsilon ,\omega )$ satisfies (\ref{sizerR}), then $N\geq N_{\varepsilon
,j}(\omega ),$ hence there can be no choice of $k\geq \phi _{\omega
}^{\varepsilon }(N)-1$ with 
\begin{equation*}
(M_{N+2}\cdot \cdot \cdot M_{N+k-L-1})^{-1}<(r_{N+1}\cdot \cdot \cdot
r_{N+k})^{-\underline{d}(1-\varepsilon )}.
\end{equation*}%
We deduce that for all $r\leq R^{1+\Phi (R)}\leq \rho (\varepsilon ,\omega )$
and for all $x\in E(\omega ),$%
\begin{equation*}
\frac{\mu (B(x,R))}{\mu (B(x,r))}\geq (M_{N+2}\cdot \cdot \cdot
M_{n-L-1})^{-1}\geq (r_{N+1}\cdot \cdot \cdot r_{n})^{-\underline{d}%
(1-\varepsilon )}\geq \left( \frac{R}{r}\right) ^{\underline{d}%
(1-\varepsilon )}.
\end{equation*}

Let $\Delta _{\varepsilon }=\bigcup_{j=1}^{\infty }\Delta _{\varepsilon
}(j) $. This is a set of measure one and we have just shown that $\underline{%
\dim }_{\Phi ^{\varepsilon }}\mu _{\omega }\geq $\underline{$d$}$%
(1-\varepsilon )$ for all $\omega \in \Delta _{\varepsilon }$. Take a
sequence $\varepsilon _{i}\rightarrow 0$ and let $\Delta
=\bigcap\limits_{i=1}^{\infty }\Delta _{\varepsilon _{i}}$. Of course, $%
\mathbb{P}(\Delta )=1$ and $\underline{\dim }_{\Phi ^{(\varepsilon
_{i})}}\mu _{\omega }\geq $\underline{$d$}$(1-\varepsilon _{i})$ for all $%
\omega \in \Delta $.

\smallskip

Now suppose $\Phi $ is any large dimension function. Given any $\varepsilon
_{i},$ we have $\Phi (x)\geq \Phi ^{(\varepsilon _{i})}(x)$ for small enough 
$x$ and therefore $\underline{\dim }_{\Phi }\mu _{\omega }\geq $ $\underline{%
\dim }_{\Phi ^{(\varepsilon _{i})}}\mu _{\omega }$ for all $\omega $. We
conclude that $\underline{\dim }_{\Phi }\mu _{\omega }\geq $\underline{$d$}$%
(1-\varepsilon _{i})$ for all $\omega \in \Delta $ and for all $i$, and
consequently, $\underline{\dim }_{\Phi }\mu _{\omega }\geq $\underline{$d$}
on $\Delta $.

\medskip

For the opposite inequality, fix $\varepsilon >0$ and choose $K_{\varepsilon
}^{\prime }$ and $F_{N}^{\prime }(\varepsilon )$ as in Lemma \ref{L2}(i) so
that $\{F_{N}^{\prime }(\varepsilon )$ i.o.$\}^{c}=\Delta _{\varepsilon
}^{\prime }$ is a set of measure one. That means, for all $\omega \in $ $%
\Delta _{\varepsilon }^{\prime }$ there are arbitrarily large $N$ with the
property that for all large enough $k,$ 
\begin{equation*}
(M_{N+1}\cdot \cdot \cdot M_{N+k})^{-1}\leq (r_{N+1}\cdot \cdot \cdot
r_{N+k)})^{-\underline{d}(1+\varepsilon )}.
\end{equation*}

Let $\Phi $ be any large dimension function. For $\omega \in $ $\Delta
_{\varepsilon }^{\prime }$, take 
\begin{equation*}
R_{N}=\tau r_{1}(\omega )\cdot \cdot \cdot r_{N}(\omega )/2.
\end{equation*}%
As $\tau /2\geq 2^{-L}\geq r_{N+1}\cdot \cdot \cdot r_{N+L}$, we have $%
R_{N}\geq r_{1}\cdot \cdot \cdot r_{N+L}$ and thus if $k\geq \phi _{\omega
}(N+L),$ then 
\begin{equation*}
\varrho _{N,k}:=r_{1}\cdot \cdot \cdot r_{N+L+k}\leq R_{N}^{1+\Phi (R_{N})}.
\end{equation*}%
Fix some Moran set of level $N$ and for each $k\geq L+\phi (N+L)$ consider a
descendent Moran set of step\ $N+k$ where at each step from $N+1$ to $N+k$
we select the child that is assigned the maximum probability. Choose $%
x=x_{N,k}\in E(\omega )$ in that descendent, so $B(x,\varrho
_{N,k})\supseteq I_{N+k}(x)$. Since the distance from $x$ to any Moran set
of level $N$, other than $I_{N}(x),$ is at least $2R_{N}$, we have $%
B(x,R_{N})\cap E(\omega )\subseteq I_{N}(x)$, and therefore%
\begin{equation*}
\frac{\mu (B(x,R_{N}))}{\mu (B(x,\varrho _{N,k}))}\leq \frac{\mu (I_{N}(x))}{%
\mu (I_{N+k}(x))}=(M_{N+1}\cdot \cdot \cdot M_{N+k})^{-1}\text{.}
\end{equation*}%
Hence there are choices of $k\geq \phi _{\omega }(N+L)$ such that 
\begin{equation*}
\frac{\mu (B(x,R_{N}))}{\mu (B(x,\varrho _{N,k}))}\leq (r_{N+1}\cdot \cdot
\cdot r_{N+k})^{-\underline{d}(1+\varepsilon )}=C_{\varepsilon }\left( \frac{%
R_{N}}{\varrho _{N,k}}\right) ^{\underline{d}(1+\varepsilon )}
\end{equation*}%
(with constant $C_{\varepsilon }=(2/\tau )^{\underline{d}(1+\varepsilon )}$%
). It follows that $\underline{\dim }_{\Phi }\mu _{\omega }\leq
(1+\varepsilon )$\underline{$d$} for every $\omega \in \Delta _{\varepsilon
}^{\prime }$ . Again, take a sequence $\varepsilon _{i}\rightarrow 0$ and
let $\Delta ^{\prime }=\bigcap\limits_{i=1}^{\infty }\Delta _{\varepsilon
_{i}}^{\prime },$ a set of full measure. Obviously, $\underline{\dim }_{\Phi
}\mu _{\omega }\leq $\underline{$d$} for all $\omega \in \Delta ^{\prime }$.

\smallskip

Combining these facts, we see that $\underline{\dim }_{\Phi }\mu _{\omega }=$%
\underline{$d$} for all large dimension functions $\Phi $ and for every $%
\omega \in \Delta $ $\bigcap \Delta ^{\prime },$ a set of measure one.

\medskip

\textbf{Computation of the almost sure upper }$\Phi $\textbf{-dimension: }%
This computation is very similar to that of the lower $\Phi $-dimension.
First, use the upper bound in (\ref{ratiobounds}), 
\begin{equation*}
\frac{\mu (B(x,R))}{\mu (B(x,r))}\leq (m_{N-L+1}\cdot \cdot \cdot
m_{n})^{-1},
\end{equation*}%
together with Lemma \ref{L1}(ii) to prove that for each $\varepsilon >0$
there is a set of full measure, $\Delta _{\varepsilon },$ such that $%
\overline{\dim }_{\Phi ^{(\mathcal{K}_{\varepsilon })}}\mu _{\omega }\leq 
\overline{d}(1+\varepsilon )$ for a suitable choice of constant $\mathcal{K}%
_{\varepsilon }$. But any large dimension function $\Phi $ dominates the
function $\Phi ^{(\mathcal{K}_{\varepsilon })}$ for small enough $x$, hence $%
\overline{\dim }_{\Phi }\mu _{\omega }\leq \overline{d}(1+\varepsilon )$ on $%
\Delta _{\varepsilon }$, as well. Then use Lemma \ref{L2}(ii), along with $%
R_{N}=r_{1}\cdot \cdot \cdot r_{N}$, $\varrho _{N,k}=\tau r_{1}\cdot \cdot
\cdot r_{N+k}/2$ and $x_{N,k}\in E(\omega )\bigcap I_{N}$ belonging to a
Moran descendent of level $N+k,$ where at each step $N+1$ to $N+k$ we choose
a Moran set where the minimal probability is assigned, to deduce that $%
\overline{\dim }_{\Phi }\mu _{\omega }\geq \overline{d}(1-\varepsilon )$ on
some set $\Delta _{\varepsilon }^{\prime }$ of full measure (independent of
the choice of $\Phi $). To complete the proof, take a sequence $\varepsilon
_{i}\rightarrow 0$ as before.

\smallskip

Of course, we can take as $\Gamma $ the intersection of the two sets of full
measure where the upper and lower $\Phi $ dimensions are the appropriate
values.
\end{proof}

\begin{remark}
Note that the proof actually shows that if $\varepsilon >0$, then (in the
notation of the proof) $\underline{\dim }_{\Phi ^{(\mathcal{K}_{\varepsilon
})}}\mu _{\omega }\geq \underline{d}(1-\varepsilon )$ and $\overline{\dim }%
_{\Phi ^{(\mathcal{K}_{\varepsilon })}}\mu _{\omega }\leq \overline{d}%
(1+\varepsilon )$ a.s.
\end{remark}

\section{Dimension Results for Small $\Phi $ \label{smallsection}}

\textbf{Terminology}: We call a dimension function $\Phi $ \textbf{small} if
it satisfies $\Phi (x)\ll \log |\log x|/|\log x|$.

The constant function $\Phi =0$ (giving the Assouad dimension) is an example
of a small $\Phi $.

\begin{notation}
Put%
\begin{equation*}
\alpha =\inf \{s:\log m_{1}/\log r_{1}\leq s\text{ a.s.}\}=\esssup \text{ }%
\log m_{1}/\log r_{1},
\end{equation*}%
and 
\begin{equation*}
\beta =\sup \{t:\log M_{1}/\log r_{1}\geq t\text{ a.s.}\}=\essinf \log
M_{1}/\log r_{1}.
\end{equation*}
\end{notation}

Obviously $\alpha \geq \beta $. Here are some other easy facts.

\begin{lemma}
\label{alphabeta}(i) If $\esssup M_{1}=1$, or $\essinf r_{1}=0$ but $\essinf %
M_{1}\neq 0$, then $\beta =0$. If $\essinf m_{1}=0,$ but $\essinf r_{1}\neq
0,$ then $\alpha =\infty $.

(ii) If $\{r_{1},m_{1}\}$ are independent random variables and $\essinf %
m_{1}=$ $0,$ then $\alpha =\infty $.
\end{lemma}

\begin{proof}
Part (i) is obvious.

(ii) Since $r_{1}(\omega )>0$ for all $\omega ,$ there must be some $%
\varepsilon >0$ such that $\mathbb{P}(r_{1}\geq \varepsilon )>0$. The
independence of $r_{1}$ and $m_{1}$ ensures that for any $q\in \mathbb{R}%
^{+},$%
\begin{align*}
\mathbb{P}(-\log m_{1} \geq -q\log r_{1}) &\geq \mathbb{P}(-\log m_{1}\geq
-q\log \varepsilon \text{ and}-\log r_{1}\leq -\log \varepsilon ) \\
&=\mathbb{P}(-\log m_{1}\geq -q\log \varepsilon \text{ })\mathbb{P}(-\log
r_{1}\leq -\log \varepsilon ) \\
&=\mathbb{P}(m_{1}\leq \varepsilon ^{q})\mathbb{P}(r_{1}\geq \varepsilon ).
\end{align*}%
But $\mathbb{P}(m_{1}\leq \varepsilon ^{q})>0$ as $\essinf m_{1}=$ $0$ and
therefore $\alpha =\infty $.
\end{proof}

\begin{theorem}
\label{small}There is a set $\Gamma $, of full measure in $\Omega ,$ with
the following properties:

(i) \underline{$\dim $}$_{\Phi }\mu _{\omega }\leq \beta $ and $\overline{%
\dim }_{\Phi }\mu _{\omega }\geq \alpha $ for all small dimension functions $%
\Phi $ and all $\omega \in \Gamma $.

(ii) If either $\essinf M_{j}>0$ or $\essinf r_{j}>0,$ then \underline{$\dim 
$}$_{\Phi }\mu _{\omega }\geq \beta $ for all dimension functions $\Phi $
and all $\omega \in \Gamma $.

(iii) If either $\essinf m_{j}>0$ or $\essinf r_{j}>0,$ then $\overline{\dim 
}_{\Phi }\mu _{\omega }\leq \alpha $ for all dimension functions $\Phi $ and
all $\omega \in \Gamma $.
\end{theorem}

Here are some immediate consequences.

\begin{corollary}
If $\essinf r_{j}>0,$ then \underline{$\dim $}$_{\Phi }\mu _{\omega }=\dim
_{L}\mu _{\omega }=\beta $ and $\overline{\dim }_{\Phi }\mu _{\omega }=\dim
_{A}\mu _{\omega }=\alpha $ for all small $\Phi $ and almost every $\omega $%
. If the number of children is bounded, then \underline{$\dim $}$_{\Phi }\mu
_{\omega }=\dim _{L}\mu _{\omega }=\beta $ for all small $\Phi $ and a.e. $%
\omega $.
\end{corollary}

\begin{proof}
\lbrack of Theorem \ref{small}] (i) We will assume $\alpha <\infty $; the
proof for $\alpha =\infty $ is similar and left for the reader. The
definition of $\alpha $ ensures that given $\varepsilon >0,$ we can choose $%
q(\varepsilon )\leq \exp (-3)$ with 
\begin{equation*}
0<q(\varepsilon )\leq \mathbb{P}\left( -\log m_{1}\geq (\alpha -\varepsilon
)(-\log r_{1})\right) .
\end{equation*}
Let 
\begin{equation*}
\Phi ^{\varepsilon }(x)=\frac{\log 2}{2\left\vert \log q(\varepsilon
)\right\vert }\frac{\log \left\vert \log x\right\vert }{\left\vert \log
x\right\vert }\text{ and }\chi _{N}^{\varepsilon }=\frac{\log (2N\mathbb{E}%
(Z_{1}))}{2\left\vert \log q(\varepsilon )\right\vert }
\end{equation*}%
and obtain the set, $\Gamma _{\varepsilon }$, of full measure from Lemma \ref%
{phi}, with the property that $\phi _{\omega }^{\varepsilon }(N)\leq \chi
_{N}^{\varepsilon }$ eventually for all $\omega \in \Gamma _{\varepsilon }$.
Set $J_{N}(\varepsilon )=\left\lceil \chi _{N}^{\varepsilon }\right\rceil $
(meaning, the next integer) and let%
\begin{equation*}
G_{N}^{\varepsilon }=\left\{ (m_{N+1}\cdot \cdot \cdot
m_{N+J_{N}(\varepsilon )})^{-1}\geq (r_{N+1}\cdot \cdot \cdot
r_{N+J_{N}(\varepsilon )})^{-(\alpha -\varepsilon )}\right\} .
\end{equation*}%
As the random variables $\{m_{i}\}$ are independent and identically
distributed, as are the random variables $\{r_{i}\},$ 
\begin{eqnarray*}
\mathbb{P}(G_{N}^{\varepsilon }) &=&\mathbb{P}\left(
\sum_{i=1}^{J_{N}(\varepsilon )}-\log m_{i}\geq (\alpha -\varepsilon
)\sum_{i=1}^{J_{N}(\varepsilon )}-\log r_{i}\right)  \\
&\geq &\mathbb{P}\left( -\log m_{i}\geq (\alpha -\varepsilon )(-\log r_{i})%
\text{ for each }i=1,...,J_{N}(\varepsilon )\right)  \\
&=&\prod_{i=1}^{J_{N}(\varepsilon )}\mathbb{P}\left( -\log m_{i}\geq (\alpha
-\varepsilon )(-\log r_{i})\right) \geq q(\varepsilon )^{J_{N}(\varepsilon
)}.
\end{eqnarray*}%
It is easy to check that $q(\varepsilon )^{J_{N}(\varepsilon )}\geq 1/N$ for
large $N$, thus if we set $N_{k}=k\log k$, then 
\begin{equation*}
\sum_{k}\mathbb{P}(G_{N_{k}}^{\varepsilon })\geq \sum_{k}q(\varepsilon
)^{J_{N_{k}}(\varepsilon )}\geq \sum_{k}\frac{1}{k\log k}=\infty \text{.}
\end{equation*}%
The fact that $\left\vert \log q(\varepsilon )\right\vert \geq 3$ ensures
that $N_{k+1}-N_{k}\geq \log (k+1)>J_{N_{k}}(\varepsilon )$ (for large $k$),
hence the events $G_{N_{k}}^{\varepsilon }$ are independent and the Borel
Cantelli lemma implies $G_{N_{k}}^{\varepsilon }$ occurs infinitely often
with probability one, say on $\Gamma _{\varepsilon }^{\prime }$.

For $\omega \in \Gamma _{\varepsilon }\bigcap \Gamma _{\varepsilon }^{\prime
}=\Delta _{\varepsilon }$, choose $\mathcal{N}_{\omega }$ such that for $%
N\geq \mathcal{N}_{\omega }$, $\phi _{\omega }^{\varepsilon }(N)\leq \chi
_{N}^{\varepsilon }$ $.$ Put $R_{N}=r_{1}(\omega )\cdot \cdot \cdot
r_{N}(\omega )$ and $\varrho _{N}=\tau $ $r_{1}\cdot \cdot \cdot
r_{N+J_{N}(\varepsilon )}/2$. As $J_{N}(\varepsilon )\geq $ $\phi _{\omega
}^{\varepsilon }(N)$, $\varrho _{N}\leq R_{N}^{1+\Phi ^{\varepsilon
}(R_{N})} $. Take any Moran set $I_{N}$ of step $N$ and consider the
descendent at step $N+J_{N}(\varepsilon )$ where the minimal probability was
chosen each time. Let $x_{N}\in E(\omega )$ belong to that descendent. The
separation property and size of $\varrho _{N}$ ensures that $B(x_{N},\varrho
_{N})\cap E(\omega )\subseteq I_{N+J_{N}(\varepsilon )}(x)$, while $%
B(x_{N},R_{N})\supseteq I_{N}(x)$. Hence 
\begin{equation*}
\frac{\mu (B(x_{N},R_{N}))}{\mu (B(x_{N},\varrho _{N}))}\geq \frac{\mu
(I_{N}(x))}{\mu (I_{N+J_{N}(\varepsilon )}(x))}=(m_{N+1}\cdot \cdot \cdot
m_{N+J_{N}(\varepsilon )})^{-1}.
\end{equation*}%
For infinitely many $N\in \{N_{k}\},$ 
\begin{equation*}
(m_{N+1}\cdot \cdot \cdot m_{N+J_{N}(\varepsilon )})^{-1}\geq (r_{N+1}\cdot
\cdot \cdot r_{N+J_{N}(\varepsilon )})^{-(\alpha -\varepsilon )}=\left( 
\frac{\tau }{2}\right) ^{\alpha -\varepsilon }\left( \frac{R}{r}\right)
^{\alpha -\varepsilon },
\end{equation*}%
thus for suitable arbitrarily small $R$, $r\leq R^{1+\Phi (R)}$ and $x\in
E(\omega ),$ we have 
\begin{equation*}
\frac{\mu (B(x,R))}{\mu (B(x,r))}\geq C_{\varepsilon }\left( \frac{R}{r}%
\right) ^{\alpha -\varepsilon }
\end{equation*}%
for a constant $C_{\varepsilon }>0$. That implies $\overline{\dim }_{\Phi
^{\varepsilon }}\mu _{\omega }\geq \alpha -\varepsilon $ for all $\omega \in
\Delta _{\varepsilon }$.

If $\Phi $ is any small dimension function, then $\Phi (x)\leq \Phi
^{\varepsilon }(x)$ for small $x$ and consequently $\overline{\dim }_{\Phi
}\mu _{\omega }\geq $ $\overline{\dim }_{\Phi ^{\varepsilon }}\mu _{\omega
}\geq \alpha -\varepsilon $ for all $\omega \in \Delta _{\varepsilon }$.
Taking a sequence $\varepsilon _{i}\rightarrow 0$, we conclude that $%
\overline{\dim }_{\Phi }\mu _{\omega }\geq \alpha $ for all $\omega \in
\Delta =\bigcap \Delta _{\varepsilon _{i}},$ a set of full measure.

\medskip

The arguments are similar for the lower dimension, but in this case we begin
with%
\begin{equation*}
0<q(\varepsilon )\leq \mathbb{P}\left( -\log M_{1}\geq (\beta +\varepsilon
)(-\log r_{1})\right)
\end{equation*}%
and define $\Phi ^{\varepsilon }$ and $\chi _{N}^{\varepsilon }$
accordingly. Put 
\begin{equation*}
G_{N}^{\varepsilon }=\left\{ \omega :(M_{N+1}\cdot \cdot \cdot
M_{N+L+J_{N+L}(\varepsilon )})^{-1}\leq (r_{N+1}\cdot \cdot \cdot
r_{N+L+J_{N+L}(\varepsilon )})^{-(\beta +\varepsilon )}\right\} .
\end{equation*}%
For large $k$, $N_{k+1}\geq N_{k}+L+J_{N_{k}+L}(\varepsilon ),$ so analogous
reasoning to the above with the Borel Cantelli lemma implies $%
G_{N_{k}}^{\varepsilon }$ occurs infinitely often with probability one.

Put $R_{N}=\tau r_{1}\cdot \cdot \cdot r_{N}/2$, $\varrho _{N}=r_{1}\cdot
\cdot \cdot r_{N+L+J_{N+L}(\varepsilon )}$, take a Moran set of step $N$ and
consider the descendent at step $N+L+J_{N+L}(\varepsilon )$ where we choose
the maximum probability each time. For $x_{N}\in E(\omega )$ belonging to
this descendent and for suitable arbitrarily large $N\in \{N_{k}\},$ we have%
\begin{eqnarray*}
\frac{\mu (B(x_{N},R_{N}))}{\mu (B(x_{N},\varrho _{N}))} &\leq &\frac{\mu
(I_{N}(x_{N}))}{\mu (I_{N+L+J_{N+L}(\varepsilon )}(x))}=(M_{N+1}\cdot \cdot
\cdot M_{N+L+J_{N+L}(\varepsilon )})^{-1}\text{ } \\
&\leq &(r_{N+1}\cdot \cdot \cdot r_{N+L+J_{N+L}(\varepsilon )})^{-(\beta
+\varepsilon )}\leq C_{\varepsilon }\left( \frac{R}{r}\right) ^{\beta
+\varepsilon },
\end{eqnarray*}%
Thus \underline{$\dim $}$_{\Phi ^{\varepsilon }}\mu \leq \beta +\varepsilon $
on a set of full measure, $\Delta _{\varepsilon }^{\prime }$. As above, we
deduce that for any small dimension function $\Phi ,$ \underline{$\dim $}$%
_{\Phi }\mu \leq \beta $ on the set of full measure $\Delta ^{\prime }=$ $%
\bigcap \Delta _{\varepsilon _{i}}^{\prime }$.

\medskip

(ii) Fix $\omega $ and assume $r\leq R^{1+\Phi (R)}\leq R$ as in (\ref%
{sizerR}), so $R/r$ $\leq (r_{N+1}\cdot \cdot \cdot r_{n})^{-1}$.

For any $x\in E(\omega ),$ Lemma \ref{BallsMoranInt} implies 
\begin{equation}
\frac{\mu (B(x,R))}{\mu (B(x,r))}\geq (M_{N+2}\cdot \cdot \cdot
M_{n-L-1})^{-1}\text{ if }n>N+L+2.  \label{1}
\end{equation}%
Regardless of $n,$ $\mu (B(x,R))/\mu (B(x,r))\geq 1$.

If $\essinf M_{j}=\delta >0$, then in either case 
\begin{equation*}
\frac{\mu (B(x,R))}{\mu (B(x,r))}\geq C_{\delta }(M_{N}\cdot \cdot \cdot
M_{n})^{-1}
\end{equation*}%
for some constant $C_{\delta }>0$. If, instead, $\essinf r_{j}=\delta >0$
and $n\leq N+L+2$, then $R/r\leq \delta ^{-(L+2)},$ while if $n>N+L+2,$ then 
$R/r\leq C_{\delta }(r_{N+2}\cdot \cdot \cdot r_{n-L-1})^{-1}$.

It follows that for any $\omega $ in the set of full measure, $\Omega
_{\varepsilon },$ where $-\log M_{i}/-\log r_{i}\geq \beta -\varepsilon $
for all $i,$ there is a constant $c>0$ such that 
\begin{equation*}
\frac{\mu (B(x,R))}{\mu (B(x,r))}\geq c\left( \frac{R}{r}\right) ^{\beta
-\varepsilon }.
\end{equation*}%
We conclude that for any $\Phi $, \underline{$\dim $}$_{\Phi }\mu \geq \beta 
$ on $\bigcap \Omega _{\varepsilon _{j}}$ where $\varepsilon _{j}\rightarrow
0$.

\medskip

(iii) This is similar to (ii) (but easier) since if $R$ and $r\leq R^{1+\Phi
(R)}$ satisfy (\ref{sizerR}), then for all $n>N,$ 
\begin{equation*}
\frac{\mu (B(x,R))}{\mu (B(x,r))}\leq \frac{\mu (I_{N-L}(x))}{\mu (I_{n}(x))}%
\leq (m_{N-L+1}\cdot \cdot \cdot m_{n})^{-1}.
\end{equation*}
\end{proof}

\begin{remark}
In Subsection \ref{smallcounter}, we see that it is possible to have $%
\overline{\dim }_{\Phi }\mu >\alpha $ and \underline{$\dim $}$_{\Phi }\mu
<\beta $ on sets of positive measure.
\end{remark}

\begin{remark}
It would be interesting to know formulas for the $\Phi $-dimensions when $%
\Phi (x)\sim \log |\log x|/|\log x|$.
\end{remark}

\section{Applications\label{applications}}

In this section we will consider some classes of examples of random Moran
sets and measures to which our results apply.

\subsection{Fixed probabilities}

\begin{example}
One obvious special case of this set up is the deterministic model where $%
T_{j}(\omega )=T,$ $r_{j}(\omega )=r\leq r_{T}$ and $(p_{j}^{(1)}(\omega
),...,p_{j}^{(T_{j})}(\omega ))=(p^{(1)},...,p^{(T)})$ for all $j$ and all $%
\omega $. Formally, this arises by choosing the probability measure to be a
single point mass measure. We clearly have $\mathbb{E(}e^{\lambda X_{1}}),$ $%
\mathbb{E(}e^{\lambda Y_{1}}),$ $\mathbb{E(}e^{\lambda Z_{1}})$ $<\infty $
for all $\lambda $ and $\mathbb{E}(X_{1})=-\log (\max p_{1}^{(j)})$, $%
\mathbb{E}(Y_{1})=-\log (\min p_{1}^{(j)})$, and $\mathbb{E}(Z_{1})=-\log
r_{1}$.

An example of this situation is to take an iterated function system on $%
\mathbb{R}^{D}$, satisfying the strong separation condition, with $T$
similarities all having contraction factor $r>0,$ and probabilities $p^{(j)}$%
. The strong separation condition ensures the $(T,r,\tau )$-separation
condition is automatically satisfied for a suitable choice of $\tau >0$.
Thus the self-similar measure has 
\begin{equation*}
\overline{\dim }_{\Phi }\mu =\frac{\log (\min p^{(j)})}{\log r}\text{ and }%
\underline{\dim }_{\Phi }\mu =\frac{\log (\max p^{(j)})}{\log r}
\end{equation*}%
for all dimension functions $\Phi $. This class of examples was also worked
out in \cite{HH}.

But we do not require the rigid structure of a self-similar set. For
instance, we need not apply the same similarities to different parents and
we need not have the same distances between children of different parents.
\end{example}

\begin{example}
Another special case is when we have a fixed set of probabilities, so $%
\mathbb{E(}e^{\lambda X_{1}}),$ $\mathbb{E(}e^{\lambda Y_{1}})$ are finite
for all $\lambda ,$ but the ratios are chosen independently and identically
distributed from some (non-trivial) set. Provided $\mathbb{E(}e^{\lambda
Z_{1}})$ $<\infty $ for small $\left\vert \lambda \right\vert ,$ then the $%
\Phi $-dimensions follow from the theorems.

If we choose the probabilities to be equal ($1/T$ if there are $T$
children), then there are constants $c,C>0$ such that for any $r\leq R,x\in
E(\omega )$ we have 
\begin{equation*}
cN_{r}(B(x,R))\leq \frac{\mu _{\omega }(B(x,R))}{\mu _{\omega }(B(x,r))}\leq
CN_{r}(B(x,R)).
\end{equation*}%
It follows that the $\Phi $-dimensions of the random measure $\mu _{\omega }$
and the random set $E(\omega )$ coincide.

For example, if we begin with $I_{0}=[0,1]\subseteq \mathbb{R}$, take $T=2$
and assume $r_{j}$ is chosen uniformly over $(0,1/2)$, then $\mathbb{E(}%
e^{\lambda Z_{1}})$ $<\infty $ for $\left\vert \lambda \right\vert <1$ and $%
\mathbb{E}(Z_{1})=1+\log 2$. Thus for large $\Phi ,$ 
\begin{equation*}
\underline{\dim }_{\Phi }E(\omega )=\overline{\dim }_{\Phi }E(\omega )=\frac{%
\log 2}{1+\log 2}\text{ a.s.}
\end{equation*}%
Since $\essinf r_{j}=0$ and $\esssup r_{j}=1/2,$ we have $\beta =0$ and $%
\alpha =1,$ thus for small $\Phi $ 
\begin{equation*}
\overline{\dim }_{\Phi }E(\omega )=1\text{ and }\underline{\dim }_{\Phi
}E(\omega )=0\text{ a.s.}
\end{equation*}
\end{example}

\subsection{Variable probabilities}

Another interesting class is when the ratios are constant, $r_{j}=r$ for all 
$j,$ but the probabilities, $(p_{j}^{(1)},...,p_{j}^{(T_{j})}),$ are iid
random variables chosen from a non-trivial probability space. In this case,
we obviously have $\mathbb{E(}e^{\lambda Z_{1}})$ $<\infty $ for all $%
\lambda $. Further, the separation condition forces $T_{j}\leq T$ for some $%
T<\infty $, so $M_{j}\geq 1/T$ and $\mathbb{E}(e^{\lambda X_{1}})<\infty $.
Thus we deduce the following results from the two theorems.

\begin{corollary}
\label{C1}(i) There is a set, $\Gamma ,$ of measure one, such that if $\Phi
(x)$ is any large dimension function, then%
\begin{equation*}
\underline{\dim }_{\Phi }\mu _{\omega }=\frac{\mathbb{E(}X_{1})}{|\log r|}=%
\frac{\mathbb{E}(|\log \max_{k}p_{1}^{(k)}|)}{|\log r|}\text{ for all }%
\omega \in \Gamma \text{.}
\end{equation*}%
Further, if there exists $A>0$ such that $\mathbb{E(}e^{\lambda
Y_{1}})<\infty $ for all $\left\vert \lambda \right\vert \leq A,$ then 
\begin{equation*}
\overline{\dim }_{\Phi }\mu _{\omega }=\frac{\mathbb{E(}Y_{1})}{|\log r|}=%
\frac{\mathbb{E}(|\log \min^{k}p_{1}^{(k)}|)}{|\log r|}\text{ for all }%
\omega \in \Gamma .
\end{equation*}

(ii) There is a set, $\Gamma ,$ of measure one, such that if $\Phi (x)$ is
any small dimension function, then%
\begin{equation*}
\underline{\dim }_{\Phi }\mu _{\omega }=\frac{\essinf|\log
\max_{k}p_{1}^{(k)}|}{|\log r|}\text{ and }\overline{\dim }_{\Phi }\mu
_{\omega }=\frac{\esssup|\log \min_{k}p_{1}^{(k)}|}{|\log r|}\text{ }\forall
\omega \in \Gamma \text{.}
\end{equation*}
\end{corollary}

\begin{example}
When $T_{j}=T$ and the probabilities $(p_{j}^{(1)},...,p_{j}^{(T)})$ are
uniformly distributed over the simplex $\mathcal{S}_{T}$ it follows that for
small $\Phi ,$ almost surely $\underline{\dim }_{\Phi }\mu =0$ and $%
\overline{\dim }_{\Phi }\mu =\infty $.
\end{example}

When $T=2$ and the probabilities are uniformly distributed, it is easy to
compute $\mathbb{E}(|\log \max p_{1}^{(k)}|)$ and $\mathbb{E}(|\log \min
p_{1}^{(k)}|)$ since $\min (p,1-p)$ and $\max (p,1-p)$ are uniformly
distributed over $(0,1/2]$ and $[1/2,1)$ respectively. However, for $T>2$
this is a more challenging combinatorial problem answered in the next
Proposition. The proof is given in the Appendix.

\begin{proposition}
\label{unifdist} Suppose the probabilities $(p_{j}^{(1)},...,p_{j}^{(T)})$
are uniformly distributed over the simplex $\mathcal{S}_{T}$. Then $\mathbb{%
E(}e^{\lambda X_{1}}),$ $\mathbb{E(}e^{\lambda Y_{1}})<\infty $ for $%
\left\vert \lambda \right\vert <1$ and%
\begin{equation*}
\mathbb{E(}X_{1})=\sum_{j=1}^{T}(-1)^{j+1}\binom{T}{j}\log j+\sum_{j=1}^{T-1}%
\frac{1}{j}\text{,\quad\ }\mathbb{E(}Y_{1})=\log T+\sum_{j=1}^{T-1}\frac{1}{j%
}.
\end{equation*}
\end{proposition}

Coupled with Corollary \ref{C1}, this gives many further examples. Here are
two.

\begin{example}
(1) Suppose $E(\omega )$ is the classical middle-third Cantor set. Let $\mu $
be the random measure where the probabilities $(p_{j},1-p_{j})$ are
independent and uniformly distributed random variables. If $\Phi (x)$ is
large, then almost surely 
\begin{equation*}
\overline{\dim }_{\Phi }\mu =\frac{1+\log 2}{\log 3}\text{, }\underline{\dim 
}_{\Phi }\mu =\frac{1-\log 2}{\log 3}.
\end{equation*}

(2) Suppose $E(\omega )$ is a Cantor-like set, but with three children and
all $r_{j}=r<1/3$. Let $\mu $ be the random measure where the probabilities $%
(p_{j}^{(1)},p_{j}^{(2)},p_{j}^{(3)})$ are independent and uniformly
distributed rv's. If $\Phi (x)$ is large, then almost surely 
\begin{equation*}
\overline{\dim }_{\Phi }\mu =\frac{3/2+\log 3}{\left\vert \log r\right\vert }%
\text{, }\underline{\dim }_{\Phi }\mu =\frac{3/2-3\log 2+\log 3}{\left\vert
\log r\right\vert }.
\end{equation*}
\end{example}

\subsection{Random 1-variable model}

Another important class of examples are the random 1-variable Moran sets and
measures described in Example \ref{1var}, hence our theorems give formulas
for the almost sure upper and lower $\Phi $-dimensions for these measures.

If we choose the probabilities associated with each of the IFS to be equal,
then the dimension of the random measure and set $E(\omega )$ coincide. In 
\cite{Tr}, Troscheit showed that $\dim _{qA}E=\dim _{H}E$ almost surely. In
fact, our results show that the upper and lower $\Phi $-dimensions for all
large $\Phi $ coincide almost surely, so even $\dim _{qA}E=\dim _{qL}E$ a.s.

If, for example, the finitely many IFS are chosen with equal likelihood,
then the dimension formulas are very simple. Assume there are a total of $N$
families of iterated function systems, where the $j$'th family consists of $%
K_{j}$ similarities, having (common) contraction factor $r_{j}$ and
probabilities $1/K_{j}$. Then there is a set, $\Gamma $, of full measure,
such that for all all $\omega \in \Gamma ,$%
\begin{equation*}
\overline{\dim }_{\Phi }E(\omega )=\text{ }\underline{\dim }_{\Phi }E(\omega
)=\frac{\sum_{i=1}^{N}\log K_{i}}{\left\vert \sum_{i=1}^{N}\log
r_{i}\right\vert }\text{ for all large }\Phi ,
\end{equation*}%
and 
\begin{equation*}
\overline{\dim }_{\Phi }E(\omega )=\max \frac{\log K_{i}}{\left\vert \log
r_{i}\right\vert },\text{ }\underline{\dim }_{\Phi }E(\omega )=\min \frac{%
\log K_{i}}{\left\vert \log r_{i}\right\vert }\text{ for all small }\Phi 
\text{.}
\end{equation*}

\subsection{Examples of strict inequality in Theorem \protect\ref{small}}

\label{smallcounter}

We conclude with examples that show we need not have $\overline{\dim }_{\Phi
}\mu =\alpha $ or \underline{$\dim $}$_{\Phi }\mu =\beta $ a.s. (in the
notation of Theorem \ref{small}).

For these examples, we will work in $\mathbb{R}$ with the initial compact
set $I_{0}=[0,1]$. Thus all the Moran sets of each step will be intervals.
We will also assume a stronger separation condition, namely that the gaps
adjacent to the Moran intervals of step $n$ which are assigned measure equal
to that of the parent interval times $m_{n}$ (resp., $M_{n})$ have length at
least $\tau r_{1}\cdot \cdot \cdot r_{n-1}$.

\begin{proposition}
\label{strongersetup}In addition to assuming $I_{0}=[0,1]\subseteq \mathbb{R}
$ and the stronger separation condition described above, we will assume that
for each $n,$ $\{m_{n+1},r_{n+1}\}$ is independent of $%
\{m_{l},r_{l}:l=1,...,n\}$. Suppose $\Phi $ is a small dimension function.

If there are constants $\theta ,c>0$ such that for all small $z>0$, $\mathbb{%
P(}m_{1}\leq z)\geq cz^{\theta }$, (resp., $\mathbb{P(}r_{1}\leq z)\geq
cz^{\theta })$, then $\overline{\dim }_{\Phi }\mu =\infty $ (resp., 
\underline{$\dim $}$_{\Phi }\mu =0$) almost surely.
\end{proposition}

\begin{proof}
Being small, we can assume $\Phi \leq H(x)\log \left\vert \log x\right\vert
/\log x|$ where $H$ decreases as $x\downarrow 0$. Let $J_{N}=\left\lceil
\chi _{N}^{(H)}\right\rceil $. For any $\omega $ in the set of full measure
where $\phi _{\omega }(N)\leq \chi _{N}^{(H)}$ eventually, put $%
R_{N}=r_{1}\cdot \cdot \cdot r_{N}$ and $\varrho _{N}=\tau r_{1}\cdot \cdot
\cdot r_{N+J_{N}}$. For large enough $N,$ $\varrho _{N}\leq R_{N}^{1+\Phi
(R_{N})}$. Take any Moran interval $I_{N}$ of step $N$ and consider the
descendent at step $N+J_{N}+1$ where the minimal probability was chosen each
time. Let $x_{N}\in E(\omega )$ belong to that descendent. The stronger
separation condition ensures that $B(x_{N},\varrho _{N})\bigcap E(\omega
)\subseteq I_{N+1+J_{N}}(x)\bigcap E(\omega ),$ so 
\begin{equation*}
\frac{\mu (B(x_{N},R_{N}))}{\mu (B(x_{N},\varrho _{N}))}\geq \frac{\mu
(I_{N}(x_{N}))}{\mu (I_{N+1+J_{N}}(x_{N}))}=(m_{N+1}\cdot \cdot \cdot
m_{N+J_{N}+1})^{-1}.
\end{equation*}

For $\gamma \in \mathbb{R}^{+},$ let 
\begin{eqnarray*}
F_{N} &=&\{\omega :(m_{N+1}\cdot \cdot \cdot m_{N+J_{N}+1})^{-1}\geq
(r_{N+1}\cdot \cdot \cdot r_{N+J_{N}})^{-\gamma }\} \\
&=&\left\{ m_{N+J_{N}+1}^{-1}\geq
\prod_{i=1}^{J_{N}}(m_{N+i}r_{N+i}^{-\gamma })\right\} .
\end{eqnarray*}%
Since $m_{1}r_{1}^{-\gamma }$ is real valued, there must be some $K$ such
that%
\begin{equation*}
0<\mathbb{P}(m_{1}r_{1}^{-\gamma }\leq K):=\delta .
\end{equation*}%
Now, 
\begin{eqnarray*}
\mathbb{P(}F_{N}) &\geq &\mathbb{P}\left( m_{N+J_{N}+1}^{-1}\geq K^{J_{N}}%
\text{ and}\prod_{i=1}^{J_{N}}(m_{N+i}r_{N+i}^{-\gamma })\leq
K^{J_{N}}\right) \\
&\geq &\mathbb{P}\left( m_{N+J_{N}+1}^{-1}\geq K^{J_{N}}\text{ and }%
(m_{N+i}r_{N+i}^{-\gamma })\leq K\text{ each }i=1,...,J_{N}\right) ,
\end{eqnarray*}%
so by independence and the hypothesis on the distribution of $m_{1},$ we have%
\begin{equation*}
\mathbb{P(}F_{N})\geq \mathbb{P(}m_{N+J_{N}+1}^{-1}\geq
K^{J_{N}})\prod_{i=1}^{J_{N}}\mathbb{P(}m_{N+i}r_{N+i}^{-\gamma }\leq K\text{
})\geq c(K^{-\theta }\delta )^{J_{N}}.
\end{equation*}%
Arguing with the Borel Cantelli lemma, as in the proof of Theorem \ref{small}%
, we deduce that for each $\gamma $ and a.a. $\omega ,$ $\overline{\dim }%
_{\Phi }\mu \geq \gamma $ and hence $\overline{\dim }_{\Phi }\mu =\infty $
a.s.

The arguments to see \underline{$\dim $}$_{\Phi }\mu =0$ are similar, but
this time begin with $R_{N}=\tau r_{1}\cdot \cdot \cdot r_{N-1}$ and $%
\varrho _{N}=r_{1}\cdot \cdot \cdot r_{N+L+J_{N+L}},$ which is dominated by $%
R_{N}^{1+\Phi (R_{N})}$ eventually. Take any interval $I_{N}$ of level $N,$
consider the descendent at level $N+J_{N}$ where the maximal probability was
chosen each time and let $x_{N}\in E(\omega )$ belong to that interval. The
stronger separation condition ensures that in this case, 
\begin{equation*}
\frac{\mu (B(x_{N},R_{N}))}{\mu (B(x_{N},\varrho _{N}))}\leq \frac{\mu
(I_{N}(x_{N}))}{\mu (I_{N+L+J_{N+L}}(x_{N}))}=(M_{N+1}\cdot \cdot \cdot
M_{N+L+J_{N+L}})^{-1}.
\end{equation*}

For $\varepsilon >0,$ choose $K>0$ such that $0<\mathbb{P}%
(M_{1}r_{1}^{-\varepsilon }>1/K):=\delta $ and set 
\begin{eqnarray*}
G_{N} &=&\{\omega :(M_{N+1}\cdot \cdot \cdot M_{N+L+J_{N+L}})^{-1}\leq
(r_{N}\cdot \cdot \cdot r_{N+L+J_{N+L}})^{-\varepsilon }\} \\
&=&\left\{ r_{N}^{\varepsilon }\leq
\prod_{i=1}^{L+J_{N+L}}(M_{N+i}r_{N+i}^{-\varepsilon })\right\} .
\end{eqnarray*}%
Then 
\begin{equation*}
\mathbb{P(}G_{N})\geq \mathbb{P(}r_{1}^{\varepsilon }\leq
K^{-(L+J_{N+L})})\delta ^{L+J_{N+L}}\geq c(K^{-\theta /\varepsilon }\delta
)^{L+J_{N+L}}.
\end{equation*}%
Again, we apply the Borel Cantelli lemma and deduce that for any $%
\varepsilon >0$ and a.a. $\omega ,$ \underline{$\dim $}$_{\Phi }\mu \leq
\varepsilon $.
\end{proof}

\begin{remark}
\label{stsep}Notice that the proof shows that if multiple children of a
given parent are assigned the minimum (or maximum) probability, it is enough
that the gaps adjacent to one of these children has the suitable size.
\end{remark}

\begin{example}
An example with $\overline{\dim }_{\Phi }\mu >\alpha $: Take $T=2$ and
choose $p_{j}$ independently and uniformly distributed over $(0,1)$. Set $%
r_{j}=m_{j}/2$ and construct the associated random Cantor set so that the $%
2^{n}$ Moran intervals at step $n$ have length $r_{1}\cdot \cdot \cdot r_{n}$%
. As $m_{j}\leq 1/2,$ all the gaps at step $n$ have length at least $%
r_{1}\cdot \cdot \cdot \cdot r_{n-1}$. Since $\mathbb{P(}m_{1}\geq z)=1-2z$,
an appeal to Proposition \ref{strongersetup} shows that $\overline{\dim }%
_{\Phi }\mu =\infty $ a.s. But, $\log m_{1}/\log r_{1}=\log m_{1}/(\log
m_{1}-\log 2)$ and $\inf \log m_{1}=-\infty $ a.s., thus $\alpha =1$.

More generally, it can be seen in the proof of Proposition \ref{unifdist},
that if there are $T$ children at each step and the probabilities are chosen
uniformly distributed over $\mathcal{S}_{T}$, then $\mathbb{P}(m_{1}\geq
z)=(1-Tz)^{T-1}$. Thus $\mathbb{P}(m_{1}\leq z)=1-(1-Tz)^{T-1}\geq cz$ for
suitable $c>0$.
\end{example}

\begin{example}
An example with $\overline{\dim }_{\Phi }\mu >\alpha $ and \underline{$\dim $%
}$_{\Phi }\mu <\beta $: Consider the special case where we take $%
r_{n}(\omega )=1/(4t)$ and $p_{n}(\omega )=(1/t,...,1/t)$ for all $n,$
whenever $\omega \in \Omega _{t}$. Formally, we can do this by (for example)
defining the discrete probability measure $\pi $ on $\Omega _{0}$ by $\pi
=c\sum_{t=2}^{\infty }t^{-2}\delta _{x_{t}}$ where $x_{t}=(1/4t,$ $%
(1/t,...,1/t))\in \Omega _{t}$ and $c=(\sum_{t\geq 2}t^{-2})^{-1}$. Thus $%
\mathbb{P}(\Omega _{t})=c/t^{2}$ and $\mathbb{E}(e^{-\log r_{1}/2})<\infty $.

Define the Moran set by beginning with $I_{0}=[0,1]$ and applying the rule
that at step $n,$ the $T_{n}$ children of each parent interval are placed
starting at the left endpoint of the parent, with gaps between them of
length $r_{1}\cdot \cdot \cdot r_{n},$ except for the final child of each
parent, which will be placed at the right end of the parent. This
construction ensures that the stronger separation condition of Proposition %
\ref{strongersetup}, as noted in Remark \ref{stsep}, is satisfied with the
right-most child. Obviously,%
\begin{equation*}
\mathbb{P}(m_{1}\leq 1/t)=\mathbb{P}(r_{1}\leq 1/4t)=\mathbb{P}(M_{1}\leq
1/t)=\mathbb{P}(\omega :T(\omega )\geq t)\geq ct^{-2},
\end{equation*}%
so Proposition \ref{strongersetup} gives $\overline{\dim }_{\Phi }\mu
=\infty $ and \underline{$\dim $}$_{\Phi }\mu =0$ almost surely. But $\alpha
=\beta =1$. Of course, $M_{n}=m_{n}=4r_{n}=1/t$ on $\Omega _{t}$ and $%
\mathbb{P(}\Omega _{t})>0$ for all $t=2,3,...$, so the essential infimum of
each of $m_{1},M_{1}$ and $r_{1}$ equals $0$.
\end{example}

For the special cases of the Assouad dimensions more can be said.

\begin{corollary}
Assume we begin with $I_{0}=[0,1]$ and the stronger separation condition as
in Proposition \ref{strongersetup}.

(i) If either $\essinf m_{1}=0$ or $\essinf r_{1}=0$, then $\dim_{A}\mu
_{\omega }=\infty $ a.s.

(ii) If either $\essinf M_{1}=0$ or $\essinf r_{1}=0$, then $\dim_{L}\mu
_{\omega }=0$ a.s.
\end{corollary}

\begin{proof}
(i) Take $R_{N}=r_{1}\cdot \cdot \cdot r_{N}$ and $\varrho _{N}=\tau
r_{1}\cdot \cdot \cdot r_{N}$. As $\tau <1$, $\varrho _{N}<R_{N}$ and with a
suitable choice of $x_{N}$ as in Proposition \ref{strongersetup}, 
\begin{equation*}
\frac{\mu _{\omega }(B(x_{N},R_{N}))}{\mu _{\omega }(B(x_{N},\varrho _{N}))}%
\geq m_{N+1}^{-1}(\omega ),
\end{equation*}%
while $R_{N}/\varrho _{N}=\tau ^{-1}$. If $\essinf m_{1}=0,$ then it follows
from the Borel Cantelli lemma that for any $\varepsilon >0$, $\mathbb{P(}%
m_{n}\leq \varepsilon $ i.o$)=1$. Thus for a.e. $\omega $ and any $\gamma
\in \mathbb{R}^{+},$ we have $m_{N+1}^{-1}\geq (R_{N}/\varrho _{N})^{\gamma
} $ for large $N$. That implies $\dim _{A}\mu _{\omega }=\infty $ a.s.

If $\essinf r_{1}=0,$ but $\essinf m_{1}\neq 0,$ then $\alpha =\infty $ and
hence Theorem \ref{small} (with $\Phi =0)$ implies $\dim _{A}\mu _{\omega
}=\infty $ a.s.

(ii) Since $M_{n}\geq 1/T_{n}\geq r_{n}$, if $\essinf M_{1}=0,$ then also $%
\essinf r_{1}=0,$ thus we may assume the latter. The argument is similar to
(i). Take $\omega $ from $\{\inf r_{1}=0\}$. For infinitely many $N,$ $%
r_{N}(\omega )<\tau $ and for such $N,$ put $R_{N}=\tau r_{1}\cdot \cdot
\cdot r_{N-1}$ and $\varrho _{N}=r_{1}\cdot \cdot \cdot r_{N}$ $<R_{N}$. For
suitable $x_{N},$ the gap assumption gives $B(x_{N},R_{N})\bigcap E(\omega
)=B(x_{N},\varrho _{N})\bigcap E(\omega )$. Thus $\mu (B(x_{N},R_{N}))/\mu
(B(x_{N},\varrho _{N}))=1$, while $R_{N}/\varrho _{N}=\tau
r_{N}^{-1}\rightarrow \infty $. It follows that $\dim _{L}\mu _{\omega }=0$
a.s.
\end{proof}

\section{Appendix: Proof of Proposition \protect\ref{unifdist}}

\begin{proof}
\lbrack of Proposition \ref{unifdist}] As $\sum_{j=1}^{T}p_{1}^{(j)}=1$, we
have $\min p_{1}^{(k)}\leq 1/T$, so for every $z\in (0,1/T),$%
\begin{equation*}
\{(p_{1}^{(1)},...,p_{1}^{(T)}):\min p_{1}^{(k)}\geq z\}=(z,...,z)+(1-Tz)%
\mathcal{S}_{T}.
\end{equation*}%
Thus $\mathbb{P}(\min p_{1}^{(k)}\geq z)=(1-Tz)^{T-1}$ and the probability
density function for $m_{1}=\min p_{1}^{(k)}$ is the function $%
f(x)=T(T-1)(1-Tx)^{T-2}$ for $0\leq x\leq 1/T.$

Using the binomial theorem, it follows that 
\begin{eqnarray*}
\mathbb{E(}Y_{1}) &=&\mathbb{E(}-\log
m_{1})=\int_{0}^{1/T}T(T-1)(1-Tx)^{T-2}(-\log x)dx \\
&=&-T(T-1)\sum_{k=0}^{T-2}(-1)^{k}T^{k}\binom{T-2}{k}\int_{0}^{1/T}x^{k}\log
x \\
&=&-T(T-1)\sum_{k=0}^{T-2}(-1)^{k}T^{k}\binom{T-2}{k}\frac{T^{-k-1}}{k+1}%
\left( -\log T-\frac{1}{k+1}\right) .
\end{eqnarray*}

Now we apply special cases of Melzak's formula (c.f. \cite[vol. 5: 1.3, 1.56]%
{Gould}):%
\begin{equation*}
\sum_{k=0}^{n}(-1)^{k}\binom{n}{k}\frac{1}{k+1}=\frac{1}{n+1}
\end{equation*}%
and 
\begin{equation*}
\sum_{k=0}^{n}(-1)^{k}\binom{n}{k}\frac{1}{(k+1)^{2}}=\frac{1}{n+1}%
\sum_{k=0}^{n}\frac{1}{k+1}.
\end{equation*}%
This gives%
\begin{equation*}
\mathbb{E(}Y_{1})=\log T+(T-1)\sum_{k=0}^{T-2}\binom{T-2}{k}\frac{(-1)^{k}}{%
(k+1)^{2}}=\log T+\sum_{k=1}^{T-1}\frac{1}{k}.
\end{equation*}

Similarly, for $\left\vert \lambda \right\vert <1,$ 
\begin{eqnarray*}
\mathbb{E(}e^{\lambda Y_{1}}) &=&T(T-1)\int_{0}^{1/T}x^{-\lambda
}(1-Tx)^{T-2}dx \\
&=&T(T-1)\sum_{k=0}^{T-2}(-1)^{k}T^{k}\binom{T-2}{k}\int_{0}^{1/T}x^{k-%
\lambda }dx \\
&=&(T-1)T^{\lambda }\sum_{k=0}^{T-2}\frac{(-1)^{k}}{k-\lambda +1}=T^{\lambda
}\prod_{k=1}^{T-1}\frac{k}{k-\lambda }.
\end{eqnarray*}%
Here the last equality is another consequence of Melzak's formula (\cite[%
vol. 5: 1.3]{Gould}): 
\begin{equation*}
\sum_{k=0}^{n}(-1)^{k}\binom{n}{k}\frac{1}{k+1-\lambda }=\frac{n!}{%
(1-\lambda )(n+\lambda -1)\cdot \cdot \cdot (2-\lambda )}\text{ for }\lambda
\neq 1,2,...,n+1.
\end{equation*}%
Thus, $\mathbb{E(}e^{\lambda Y_{1}})$ is finite if $\left\vert \lambda
\right\vert <1$.

We have $\mathbb{E(}e^{\lambda X_{1}})<\infty $ since $\max p_{1}^{(k)}\geq
1/T$. Lastly, we compute $\mathbb{E(}X_{1})$. It was proven in \cite{Holst}
that 
\begin{equation*}
\mathbb{P}(M_{1}\leq z)=\mathbb{P}\left( \max_{k=1,...,T}p_{1}^{(k)}\leq
z\right) =1+\sum_{k=1}^{T}(-1)^{k}\binom{T}{k}(1-kz)_{+}^{T-1}:=F(z)
\end{equation*}%
where $x_{+}=\max \{x,0\}$. Note that $(1-kz)_{+}=0$ if $z\geq 1/k$.

Since $M_{1}\geq 1/T$ with equality only on a set of measure zero,
integration by parts and the binomial theorem give%
\begin{eqnarray*}
\mathbb{E(}X_{1}) &=&\mathbb{E(}-\log M_{1})=\int_{1/T}^{1}-F^{\prime
}(x)\log (x)dx=\int_{1/T}^{1}\frac{F(x)}{x}dx \\
&=&\int_{1/T}^{1}\frac{dx}{x}+\sum_{k=1}^{T}(-1)^{k}\binom{T}{k}%
\int_{1/T}^{1/k}\frac{(1-kx)^{T-1}}{x}dx \\
&=&\log T+\sum_{k=1}^{T}(-1)^{k}\binom{T}{k}\sum_{j=0}^{T-1}(-k)^{j}\binom{%
T-1}{j}\int_{1/T}^{1/k}x^{j-1}dx.
\end{eqnarray*}%
Evaluating the integrals gives%
\begin{eqnarray*}
\mathbb{E(}X_{1}) &=&\log T+\sum_{k=1}^{T}(-1)^{k}\binom{T}{k}(\log T-\log k)
\\
&&+\sum_{k=1}^{T}(-1)^{k}\binom{T}{k}\sum_{j=1}^{T-1}\binom{T-1}{j}%
(-k)^{j}\left( \frac{k^{-j}-T^{-j}}{j}\right) \\
&=&\log T+\sum_{k=1}^{T}(-1)^{k}\binom{T}{k}(\log T-\log k)+A+B
\end{eqnarray*}%
where 
\begin{equation*}
A=\left( \sum_{k=1}^{T}(-1)^{k}\binom{T}{k}\right) \left( \sum_{j=1}^{T-1}%
\binom{T-1}{j}\frac{(-1)^{j}}{j}\right)
\end{equation*}%
and 
\begin{equation*}
B=-\sum_{k=1}^{T}(-1)^{k}\binom{T}{k}\sum_{j=1}^{T-1}\binom{T-1}{j}(-1)^{j}%
\frac{k^{j}T^{-j}}{j}.
\end{equation*}

Another application of the binomial theorem shows that%
\begin{equation}
1+\sum_{k=1}^{T}(-1)^{k}\binom{T}{k}=0.  \label{BI}
\end{equation}%
Together with the combinatorial identity (c.f. \cite[vol. 5: 1.4]{Gould})%
\begin{equation*}
\sum_{j=1}^{T-1}\binom{T-1}{j}\frac{(-1)^{j}}{j}=-\sum_{n=1}^{T-1}\frac{1}{n}%
,
\end{equation*}%
this proves $A=\sum_{n=1}^{T-1}1/n$. Euler's finite difference formula (c.f. 
\cite[vol. 4: 10.1]{Gould}) implies%
\begin{equation*}
\sum_{k=1}^{T}(-1)^{k}\binom{T}{k}k^{j}=0\text{ for }T\geq 2\text{,
\thinspace }1\leq j\leq T-1\text{.}
\end{equation*}%
Changing the order of the summation and applying this formula shows that 
\begin{equation*}
B=\left( \sum_{j=1}^{T-1}\binom{T-1}{j}\frac{(-1)^{j+1}}{T^{j}j}\right)
\left( \sum_{k=1}^{T}(-1)^{k}\binom{T}{k}k^{j}\right) =0.
\end{equation*}%
Using (\ref{BI}), we conclude that 
\begin{equation*}
\mathbb{E(}X_{1})=\sum_{k=1}^{T}(-1)^{k+1}\binom{T}{k}\log k+\sum_{k=1}^{T-1}%
\frac{1}{k}.
\end{equation*}
\end{proof}

\end{document}